\documentclass[letterpaper, 10 pt, conference]{ieeeconf}  

\IEEEoverridecommandlockouts                              
\usepackage{graphics} 
\graphicspath{{./Figures/}}
\usepackage{epsfig} 
\usepackage{times} 
\usepackage{amsmath} 
\usepackage{amssymb}  
\usepackage{subcaption}
\usepackage{wrapfig}
\usepackage{comment}
\usepackage{lipsum}
\usepackage{booktabs} 
\usepackage{multirow}

\renewcommand{\d}{\mathrm{d}}
\newcommand{\E}{\mathbb{E}}
\renewcommand{\P}{\mathbb{P}}

\newcommand{\Q}{\mathbb{Q}}
\renewcommand{\S}{\mathbb{S}}
\newcommand{\B}{\mathbb{B}}
\newcommand{\G}{\mathbb{G}}
\newcommand{\W}{\mathbb{W}}

\newcommand{\R}{\mathbb{R}}

\newcommand{\cvar}{\operatorname{CVaR}}
\newcommand{\vct}[1]{\boldsymbol{#1}}

\newtheorem{thm}{Theorem}

\newtheorem{problem}{Problem}
\newtheorem{defn}{Definition}
\newtheorem{rem}{Remark}
\newtheorem{cor}{Corollary}

\title{\LARGE \bf
Distributionally Robust Density Control with\\ Wasserstein Ambiguity Sets
}

\author{Joshua Pilipovsky and Panagiotis Tsiotras
\thanks{J. Pilipovsky is a PhD student at the School of Aerospace Engineering, Georgia Institute of Technology, Atlanta, GA30332-0150, USA. Email: jpilipovsky3@gatech.edu}%
\thanks{P. Tsiotras is the David \& Lewis Chair and Professor at the School of Aerospace Engineering and the Institute for Robotics \& Intelligent Machines, Georgia Institute of Technology, Atlanta, GA30332-0150, USA. Email: tsiotras@gatech.edu}%
}

\begin{document}

\maketitle
\thispagestyle{empty}
\pagestyle{empty}

\begin{abstract}

Precise control under uncertainty requires a good understanding and characterization of the noise affecting the system.
This paper studies the problem of steering state distributions of dynamical systems subject to partially known uncertainties.
We model the distributional uncertainty of the noise process in terms of Wasserstein ambiguity sets, which, based on recent results, have been shown to be an effective means of capturing and propagating uncertainty through stochastic LTI systems.
To this end, we propagate the distributional uncertainty of the state through the dynamical system, and, using an affine feedback control law, we steer the ambiguity set of the state to a prescribed, terminal ambiguity set.
We also enforce distributionally robust CVaR constraints for the transient motion of the state so as to reside within a prescribed constraint space.
The resulting optimization problem is formulated as a semi-definite program, which can be solved efficiently using standard off-the-shelf solvers.
We illustrate the proposed distributionally-robust framework on a quadrotor landing problem subject to wind turbulence.

\end{abstract}

\section{INTRODUCTION}
When controlling a dynamical system affected by noise, one needs to be able to discern the statistical properties of the exogenous disturbances acting on the system.
When such a characterization is unknown, or is only approximately known, care must be taken to ensure robust performance of the system under a range of uncertainties that can potentially affect the system.
Indeed, if, for example, a control designer naively assumes a normally distributed noise process, the resulting control law may severely underestimate the probability of violating the constraints or it may fail to reach a given desired terminal state \cite{JP_DR}.

To this end, we would like to systematically and tractably solve a stochastic optimal control problem that can not only control the dispersion of system states to a prescribed terminal distribution, but also steer the uncertainty of this dispersion for all disturbances sufficiently close to the true disturbance acting on the system.
The theory of covariance control, originally introduced in the 80's with works of Hotz and Skelton \cite{HS2} solved the problem of steering the first two moments of the state distribution in the infinite horizon setting.
In recent years, this theory has been extended to the finite-horizon \cite{EB1, Max1} setting, as well as extensions involving chance constraints \cite{EB2, Max2, exact_CS_2}, partially observed systems \cite{OFCS, Pili8}, and data-driven scenarios \cite{JP_DDCS} under the term covariance steering (CS) to emphasize the finite-horizon problem formulation.
This framework has been successfully applied to a plethora of problems of interest, including spacecraft rendezvous \cite{Pili4}, powered-descent guidance \cite{PDG_Jack}, interplanetary trajectory optimization \cite{JoshJack, spacecraft_Oguri, spacecraft_Benedikter}, aggressive driving \cite{CSMPC_driving}, and other pertinent applications.
The baseline theory is mathematically tractable and elegant, however it assumes Gaussian noise entering the system, as well as boundary Gaussian distributions for the initial and terminal states.
Recent extensions have relaxed these limiting assumptions and have solved CS problems with more general noise models, such as Gaussian random fields \cite{CS_GRF}, martingale processes \cite{CS_martingale}, and multiplicative noise \cite{CS_multiplicative}.

All of these extensions, albeit successful, have assumed exact knowledge of the noise model affecting the system, which is unrealistic in practice.
We are rarely fully aware of the disturbances acting on the system, and at best we can characterize partial statistical information from collected data, e.g., the first two moments.
As such, it is fruitful to consider the problem of steering the distribution of the state under distributional uncertainty in the noise model.
A natural framework to accomplish this goal is to the model the noise as residing in an \textit{ambiguity set}, which is characterized by a whole family of distributions that the noise can follow.
The goal, then, is to optimize the control law and satisfy constraints under the \textit{worst-case} disturbance that nature imposes within the allowable ambiguity set.
The work in \cite{JP_DR} has solved this problem by characterizing the distributional noise uncertainty as a Chebyshev ambiguity set, which is a family of distributions that have common first two moments, and by tractably enforcing chance constraints using concentration inequalities.
This ambiguity set, however, is still quite limited in its expressivity due to the assumption of common moments.

Recently, there has been a great promise in capturing distributional uncertainty via Wasserstein ambiguity sets, which are defined through the natural Wasserstein metric on probability spaces.
Indeed, \cite{DRO_Kuhn_risk, DRO_Kuhn_empirical} has shown that distributionally robust optimization (DRO) over Wasserstein ambiguity sets is tractable in cases where the nominal distribution is either empirical or elliptical. 
In the context of stochastic control, the works in \cite{aolaritei2023distributional, capture_propagate_control_OT} have outlined a general framework for capturing distributional noise uncertainty through empirical data collected, and have provided a procedure to propagate Wasserstein ambiguity sets through stochastic LTI systems, which is analytically exact under some mild assumptions.
Subsequent works have applied this framework to design optimal open-loop controllers while satisfying conditional value-at-risk (CVaR) constraints, as well as in the context of model-predictive control \cite{DR_DeePC}.

Our contributions are as follows.
To the best of our knowledge, this is the first work that solves the open problem \cite{aolaritei2023distributional} of optimizing over both open-loop and \textit{feedback} controllers for distributionally-robust optimal control problems, which is accomplished using an affine state feedback control law using established techniques from the CS literature.
Secondly, we show that it is possible to \textit{steer} the distributional state uncertainty to a desired terminal ambiguity set, thereby controlling the dispersion of system states under all possible noise realizations within the Wasserstein ambiguity set.
Lastly, we apply the proposed framework to the problem of a quadrotor landing subject to wind turbulence.

\section{NOTATION}

We assume a common probability space $(\Omega,\mathcal{F},\P)$ for all random objects.
Real valued-vectors are denoted by lowercase letters, $u\in\R^{m}$, matrices are denoted by upper-case letters, $V\in\R^{n\times M}$, and random vectors are denoted by boldface, $\vct x\in\R^{n}$.
The space of probability distributions over $\R^{d}$ with finite $q$th moment is denoted by $\mathcal{P}_{q}(\R^{d})$.
Given $\P,\Q\in\mathcal{P}_{q}(\R^{d})$, we denote by $\P\otimes\Q$ their product distribution and by $\P^{\otimes N}$ the $N$-fold product distribution $\P\otimes\cdots\otimes\P$ with $N$ terms.
Given a matrix $A\in\R^{m\times d}$, the pushforward of $\P$ is given by $A_{\#}\P$ and is defined by $(A_{\#}\P(\mathcal{B})) = \P(A^{-1}(\mathcal{B}))$, for all Borel sets $\mathcal{B}\subset\R^{m}$.
We denote by $\delta_{x}$ the Dirac delta distribution concentrating unit mass at the atom $x\in\R^{n}$.
The convolution of $\P$ and $\delta_{x}$ is denoted by $\delta_{x} * \P$, and is defined by $(\delta_{x} * \P)(\mathcal{B}) = \P(\mathcal{B} - x)$.
With slight abuse of notation, the operator $\|\cdot\|$ denotes the Euclidean norm for vectors and the spectral norm for matrices.
The Moore-Penrose pseudoinverse of a matrix $A$ is denoted by $A^\dagger$.
Lastly, for any $t\in\mathbb{Z}_{+}$, we set $[t] = \{0,\ldots,t\}$.

\section{PROBLEM STATEMENT}
Consider the discrete-time, stochastic linear dynamics system
\begin{equation}~\label{eq:dynamics}
    \vct x_{k+1} = A \vct x_{k} + B \vct u_k + D \vct w_{k}, \quad \forall k\in[N-1],
\end{equation}
with states $\vct x_k\in\R^{n}$, control inputs $\vct u_k\in\R^{m}$, and process noise sequence $\{\vct w_k\}_{k\in[N-1]}\subset\R^{d}$ that is neither identically nor necessarily independently distributed.
In this work, we assume the system model $\{A, B, D\}$ is known.
The noise process $\{\vct w_k\}_{k\in[N-1]}$, on the other hand, is unknown but belongs to an ambiguity set $\mathcal{W}$, which is defined rigorously in the following two definitions.
\begin{defn}[\cite{DRO_Kuhn_risk}]
    \label{defn:structural_ambiguity_set}
    A structural ambiguity set $\mathcal{S}$ is a subset of $\mathcal{P}_{2}(\R^{d})$ that is closed under positive semidefinite affine pushforwards, that is, for any $\Q\in\mathcal{S}$ and any affine transformation $f:\R^{d}\rightarrow\R^{d}$ of the form $f(\xi) = A\xi + b$, for some $A\succeq 0$ and $b\in\R^{d}$, we have $\Q \circ f^{-1} \in \mathcal{S}$.
\end{defn}

Some examples of structural ambiguity sets include the set of symmetric distributions, unimodal distributions, log-concave distributions, and elliptical distributions.
Of course, the entire probability space $\mathcal{P}_{2}(\R^{d})$ is trivially a structural ambiguity set.
\begin{defn}
    \label{defn:Wasserstein_ambiguity_set}
    The Wasserstein ambiguity set of radius $\varepsilon$ with transportation cost $c$ centered at the nominal distribution $\P$ is defined by
    \begin{equation}~\label{eq:ambiguity_set}
        \B_{\varepsilon, p}^{c}(\P) = \{\Q\in\mathcal{S} : \W_{p}^{c}(\Q,\P) \leq \varepsilon\},
    \end{equation}
    with respect to the type-$p$ Wasserstein metric
    \begin{equation}~\label{eq:wasserstein_distance}
        \W_{p}^{c}(\P,\P') \triangleq \left(\inf_{\pi\in\Pi(\P,\P')}\int_{\R^{d}\times\R^{d}}c(\xi, \xi')^{p} \ \pi(\d \xi, \d \xi')\right)^\frac{1}{p},
    \end{equation}
    where $\Pi(\P,\P')$ denotes the set of all joint probability distributions of $\xi\in\R^{d}$ and $\xi'\in\R^{d}$ with marginals $\P$ and $\P'$, respectively.
\end{defn}

In what follows, we will assume no structural information on the noise (i.e., $\mathcal{S} = \mathcal{P}_{2}(\R^{d})$), and we will work with the type-2 Wasserstein distance ($p = 2$) and the Euclidean norm transportation cost, i.e., $c = \|\cdot\|$.
For simplicity, we denote $\B_{\varepsilon}^{\|\cdot\|}\triangleq \B_{\varepsilon,2}^{\|\cdot\|}$ and $\W \triangleq \W_{2}^{\|\cdot\|}$.
A customary way to construct the noise ambiguity set is by defining the nominal distribution as $\hat\P_{w} = \frac{1}{T}\sum_{i=1}^{T}\delta_{\hat{w}^{(i)}}$, where $\hat{w}^{(i)} \triangleq [(\hat w_0^{(i)})^\intercal,\ldots, (\hat w_{N-1}^{(i)})^\intercal]^\intercal\in\R^{Nd}$ is a noise realization sampled from the underlying \textit{true} distribution $\P_{w}$.
It can be shown that by choosing a suitable radius $\varepsilon(T, \beta)$, the true distribution lies in the ball $\B_{\varepsilon}^{\|\cdot\|}(\hat\P_{w})$ with probability $1-\beta$ \cite{Wasserstein_convergence}.
In the present work, however, we assume that the noise sequence ambiguity set $\mathcal{W}$ is centered on a zero-mean normal distribution $\hat\P_{w} = \mathcal{N}(0,\Sigma_{w})$ with noise covariance matrix $\Sigma_{w} \in\R^{Nn} \succ 0$ and radius $\varepsilon > 0$.
\begin{rem}
    \label{rem:noise_ambiguity_sets}
    It is possible, and in fact customary, to define the noise ambiguity set for an individual disturbance $w_k$ via $\B_{\varepsilon}^{\|\cdot\|}(\hat\P_{w})$.
    Assuming the noise is i.i.d., then it can be shown that the ambiguity set for the entire noise \textit{sequence} is $\B_{N\varepsilon}^{\|\cdot\|}(\hat\P_{w}^{\otimes N})$.
    However, in this work we choose to define the ambiguity set directly in terms of the disturbance sequence to include cases where the noise terms are not independent of one another, thus prohibiting us from writing the joint disrtibution of $\vct w$ as an $N$-fold product distribution.
\end{rem}

We assume that the initial state $x_0 = x_i$ is deterministic, which implies that all the distributional uncertainty in the state results from the noise ambiguity set $\mathcal{W} = \B_{\varepsilon}^{\|\cdot\|}(\hat\P)$.
We define the set $\pi$ of \textit{admissible} control inputs as the set of control sequences $\{\vct u_k\}_{k\in[N]}$ where the input $\vct u_k$ is an affine function of the state.
Further, we define the \textit{nominal} state as the deterministic part of the state governed by the nominal dynamics
\begin{equation}
    \label{eq:nominal_dynamics}
    \bar{x}_{k+1} = A x_{k} + B \bar{u}_k,
\end{equation}
where $\bar{u}_{k}\in\R^{m}$ is the nominal control, and we define the \textit{error} state $\tilde{\vct x}_{k} \triangleq \vct x_{k} - \bar{x}_k$, which obeys the dynamics
\begin{equation}
    \label{eq:error_dynamics}
    \tilde{\vct x}_{k+1} = A \tilde{\vct x}_k + B \tilde{\vct u}_k + D\vct w_k,
\end{equation}
where $\tilde{\vct u}_k$ is the error control.
\begin{rem}
    \label{rem:CS_comparison}
    The nominal state as defined in this work can no longer be associated with the \textit{mean} state $\E_{\P_k}[\vct x_k]$, as is customarily done in the CS literature \cite{exact_CS_2}.
    In fact, the expectation of the state cannot even be computed because the underlying state distribution is ambiguous by definition.
\end{rem}

The goal is to steer to a terminal ambiguity set $\S_{f} \triangleq \B_{\delta}^{\|\cdot\|}(\P_{f})$, where $\delta > 0$ is a given, desired terminal radius, and $\P_{f} = \mathcal{N}(\mu_{f},\Sigma_{f})$ is the desired terminal center distribution, while minimizing the distributionally-robust objective function
\begin{equation}~\label{eq:DR_cost}
    \mathcal{J} = \beta\sum_{k=0}^{N-1} \|\bar{u}_k\| + \max_{\P\in\mathcal{W}} \E_{\P}\left[\sum_{k=0}^{N-1} \tilde{\vct x}_k^\intercal Q_k \tilde{\vct x}_k + \tilde{\vct u}_k^\intercal R_k \tilde{\vct u}_k\right],
\end{equation}
where $Q_k\succeq 0$ and $R_k\succ 0$ represent the state and input cost weights, respectively, and $\beta > 0$ denotes the weight of the nominal control.
Lastly, we would also like to enforce distributionally-robust constraints on the trajectory of the state along the planning horizon.

Letting the state constraint space be the polyehdron $\mathcal{X} \triangleq \{x : \max_{j\in[J]} \alpha_j^\intercal x + \beta_j \leq 0\}$, the traditional way of enforcing probabilistic constraints is to enforce chance constraints, which limit the probability of violating the constraints to be smaller than some prescribed risk $\gamma$ \cite{Ono_chance_constraints}.
It is well-known, however, that Value-at-Risk (VaR) constraints are \textit{not} convex, and are only exactly tractable when the underlying state distribution is normal; otherwise, they are approximated using concentration inequalities \cite{tractable_CC}.
In this work, we choose the CVaR risk measure, which is defined as follows.
\begin{defn}
    \label{defn:CVaR}
    Given $f:\R^{n}\rightarrow\R$ and a random variable $\vct x\sim\P$ on $\R^{n}$, the CVaR of $f(\vct x)$ at the quantile $1-\gamma$ is
    \begin{equation}~\label{eq:CVAR}
        \cvar_{1-\gamma}^{\P}(f(x)) = \inf_{\tau\in\R} \ \left(\tau + \frac{1}{\gamma}\E_{\P}[\max\{0,f(x)-\tau\}]\right).
    \end{equation}
\end{defn}
The CVaR of a random variable is by definition convex \cite{tractable_CC}, implicitly satisfies the VaR constraint, and mitigates the effects of extreme ``black swan" events by reducing the tail probability of violating the constraints.
To this end, we enforce the distributionally-robust CVaR (DR-CVaR) constraints
\begin{equation}~\label{eq:DR_CVaR}
    \sup_{\P_k\in\S_{k}} \cvar_{1-\gamma}^{\P_k}\left(\max_{j\in[J]}\alpha_j^\intercal x_k + \beta_j\right) \leq 0, \quad \forall k \in [N],
\end{equation}
where $\S_k$ denotes the ambiguity set of the state at time step $k$.
In summary, the distributionally-robust density steering (DR-DS) problem is defined as follows.
\begin{problem}
    \label{prob:DDDS}
    For a given initial state $x_0$, find an admissible control sequence $\{\vct u_k\}_{k\in[N]}\in\pi$ that minimizes the DR cost functional \eqref{eq:DR_cost} subject to the dynamics \eqref{eq:dynamics}, noise ambiguity set $\mathcal{W}$ and DR-CVaR constraints \eqref{eq:DR_CVaR}, such that the terminal distributional uncertainty in the state satisfies $\S_{N} \subseteq \S_{f}$.
\end{problem}

\section{PROBLEM REFORMULATION}
We begin by first reformulating the dynamics \eqref{eq:dynamics} into a more amenable form for analysis.
To this end, define the augmented state, control, and disturbance vectors $\vct x \triangleq [\vct x_0^\intercal,\ldots, \vct x_{N}^\intercal]^\intercal\in\R^{(N+1)n}, \vct u \triangleq [\vct u_0^\intercal,\ldots,\vct u_{N-1}^\intercal]^\intercal\in\R^{Nm}, \vct w \triangleq [\vct w_0^\intercal,\ldots, \vct w_{N-1}^\intercal]^\intercal\in\R^{Nd}$, respectively, which obey the augmented linear system
\begin{equation}
    \label{eq:augmented_system}
    \vct x = \mathcal{A}x_0 + \mathcal{B} \vct u + \mathcal{D} \vct w,
\end{equation}
for appropriate matrices $\mathcal{A},\mathcal{B},\mathcal{D}$ \cite{Max1}.
We consider the affine state feedback control law $\vct u_k = K_k \tilde{\vct x}_k + v_k$, where $v_k\in\R^{m}$ is the feed-forward control and $K_k\in\R^{m\times n}$ is the feedback gain.
Defining the augmented feed-forward control $v \triangleq [v_0^\intercal,\ldots,v_{N-1}^\intercal]^\intercal\in\R^{Nm}$ and augmented feedback gain matrix $K\in\R^{Nm\times (N+1)n}$, and using the state decomposition in \eqref{eq:nominal_dynamics}-\eqref{eq:error_dynamics}, the dynamics \eqref{eq:augmented_system} become
\begin{equation}~\label{eq:augmented_system_control}
    \begin{aligned}
        \bar{x} &= \mathcal{A} x_0 + \mathcal{B} v, \\
        \tilde{\vct x} &= (I - \mathcal{B}K)^{-1}\mathcal{D} \vct w.
    \end{aligned}
\end{equation}
Since $K$ is block lower-triangular and $\mathcal{B}$ is strictly block lower-triangular, it follows that that the matrix $I - \mathcal{B} K$ is invertible.
Furthermore, following \cite{JoshJack}, we define the new decision variable $L\triangleq K(I - \mathcal{B}K)^{-1}$, from which it can be shown that $I + \mathcal{B}L = (I-\mathcal{B}K)^{-1}$, and the original gains can be recovered from $K = L(I + \mathcal{B}L)^{-1}$ following the same logic.
As a result, the error state dynamics become
\begin{equation}
    \label{eq:error_dynamics_change_of_vars}
    \tilde{\vct x} = (I + \mathcal{B}L)\mathcal{D}\vct w.
\end{equation}
Given $\vct w\in\mathcal{W}$, it follows that the distributional uncertainty in the error state results from the linear transformation $\tilde{\S}_{k} = (\tilde{L}_{k})_{\#}\B_{\varepsilon}^{\|\cdot\|}(\hat\P)$, with $\tilde{L}_k \triangleq E_k(I + \mathcal{B}L)\mathcal{D}$, where $E_k\in\R^{n\times (N+1)n}$ is a matrix that isolates the $k$th state element from $\vct x$.
To this end, we now state a result on the propagation of ambiguity sets via linear transformations \cite{aolaritei2023distributional}.
\begin{thm}~\label{thm:linear_pushforward}
    Let $\P\in\mathcal{P}(\R^{d})$, and consider the linear transformation defined by the matrix $A\in\R^{m\times d}$.
    Moreover, let $c:\R^{d}\rightarrow\R_{\geq 0}$ be orthomonotone\footnote{That is, $c(x_1 + x_2) \geq c(x_1)$ for all $x_1,x_2\in\R^{d}$ satisfying $x_1^\intercal x_2 = 0$.}.
    Then,
    \begin{equation}
        A_{\#}\B_{\varepsilon}^{c}(\P) \subseteq \B_{\varepsilon}^{c\circ A^\dagger}(A_{\#}\P).
    \end{equation}
    Moreover, if the matrix $A$ is full row-rank, then
    \begin{equation}
        A_{\#}\B_{\varepsilon}^{c}(\P) = \B_{\varepsilon}^{c\circ A^\dagger}(A_{\#}\P),
    \end{equation}
    with $A^\dagger = A^\intercal (AA^\intercal)^{-1}$.
\end{thm}

Since $\tilde{L}_k\in\R^{n\times Nd}$, where $Nd \gg n$, in most cases of interest, it is safe to assume that $\tilde{L}_k$ is full row-rank without loss of generality.
Noting that the nominal state is simply a delta distribution in the probability space, the distributional uncertainty in the state at time step $k$ becomes
\begin{equation}
    \label{eq:dist_uncertainty_state}
    \S_k = \delta_{\mathcal{A}x_0 + \mathcal{B}v}*\B_{\varepsilon}^{\|\cdot\|\circ \tilde{L}_k^\dagger}\big((\tilde{L}_k)_{\#}\hat\P_{w}\big),
\end{equation}
defined on the support $\R^{n}$.
\begin{rem}
    The interpretation of \eqref{eq:dist_uncertainty_state} is that the feedback gain $L$ affects both the \textit{shape} of the center distribution as well as the \textit{size} of the ambiguity set, while the open-loop term $v$ controls the \textit{position} of the center distribution in $\mathcal{P}(\R^{n})$.
    This is a direct generalization of the traditional CS literature, where the open-loop controls the mean state, while the feedback controls the covariance of the state.
\end{rem}

In the next section, we tractably formulate the DR-CVaR constraints \eqref{eq:DR_CVaR} using the exact ambiguity set \eqref{eq:dist_uncertainty_state} and techniques from DRO.
\subsection{DR-CVaR Constraints}
To make the CVaR constraints \eqref{eq:DR_CVaR} tractable, we should use our knowledge of the ambiguity set $\S_{k}$, defined in terms of its Gaussian reference distribution $\hat\P_{k} = \mathcal{N}(0, \tilde{L}_k \Sigma_{w} \tilde{L}_k^\intercal)$, and transportation cost according to \eqref{eq:dist_uncertainty_state}.
The work in \cite{DRO_Kuhn_risk} tractably computes the DR-CVaR of piece-wise linear functions, while the work in \cite{DRO_Kuhn_ML} tractably computes the DR cost of expectations of general piece-wise quadratic functions.
In both cases, however, it can be shown that the resulting convex programs are nonlinear in the feedback gain $L$, which makes them intractable from a computational standpoint.
We thus leave it as an open problem to tractably formulate \textit{joint} DR-CVaR constraints for a polyhedral constraint space with a nominal Gaussian distribution whose covariance is parameterized by the feedback gain decision variables.

Instead, we consider an alternative where we wish to enforce the DR-CVaR constraints for each \textit{side} of the polytope along the planning horizon, that is,
\begin{equation}
    \label{eq:DR_CVAR_individual}
    \sup_{\P_k\in\S_k} \cvar_{1-\gamma_{jk}}^{\P_k}(\alpha_{j}^\intercal x_k + \beta_j) \leq 0, \ \forall j\in[J], \ \forall k\in[N].
\end{equation}
In essence, at each time step, we split up the joint risk $\gamma$ to individual risks $\gamma_{jk}$ of violating the DR-CVaR constraints along each half space and for each time step.
This now becomes the DR-CVaR of a \textit{linear} function, which we will show is SDP representable and linear in the decision variables $(v, L)$.
First, however, we need to define the notion of a \textit{Gelbrich} ambiguity set.
\begin{defn}
    \label{defn:Gelbrich_ambiguity_set}
    The Gelbrich ambiguity set of radius $\varepsilon$ centered at a mean-covariance pair $(\mu,\Sigma)$ is given by
    \begin{equation}
        \label{eq:gelbrich_set}
        \mathcal{G}_{\varepsilon}(\mu,\Sigma) = \{\Q\in\mathcal{P}(\R^{d}) : (\E_{\Q}[\xi], \mathrm{Cov}_{\Q}[\xi])\in\mathcal{U}_{\varepsilon}(\mu,\Sigma)\},
    \end{equation}
    where $\mathcal{U}_{\varepsilon}(\mu,\Sigma)$ is an uncertainty set in the space of mean vectors and covariance matrices, defined as
    \begin{equation}
        \mathcal{U}_{\varepsilon}(\hat\mu,\hat\Sigma) = \{(\mu,\Sigma)\in\R^{d}\times \mathbb{S}_{+}^{d} : \G((\mu, \Sigma), (\hat\mu,\hat\Sigma)) \leq \varepsilon\},
    \end{equation}
    where
    \begin{align}
        \G((\mu_1,\Sigma_1),(\mu_2,\Sigma_2)) &\triangleq  \|\hat\mu - \mu\|^{2} + \nonumber \\
        &\mathrm{tr}\left[\hat\Sigma + \Sigma - 2\left(\hat\Sigma^{\frac{1}{2}}\Sigma\hat\Sigma^{\frac{1}{2}}\right)^{\frac{1}{2}}\right],
    \end{align}
    is the Gelbrich distance between two mean-covariance pairs.
\end{defn}
\begin{thm}[\cite{DRO_Kuhn_risk}]~\label{thm:Gelbrich_inclusion}
    \label{thm:gelb_was_equivalence}
    If the nominal distribution $\hat\P$ has mean $\hat\mu\in\R^{d}$ and covariance matrix $\hat\Sigma\succeq 0$, then we have $\B_{\varepsilon}^{\|\cdot\|}(\hat\P) \subseteq \mathcal{G}_{\varepsilon}(\hat\mu,\hat\Sigma)$.
    In addition, if $\mathcal{S}$ is the structural ambiguity set generated by $\hat\P$ and if $\hat\Sigma \succ 0$, then the inclusion becomes an equality.
\end{thm}

Since the Gelbrich ambiguity set constitutes an \textit{outer} approximation of the associated Wasserstein ambiguity set (under the 2-norm transportation cost), satisfaction of Gelbrich DR-CVaR constraints implies satisfaction of Wasserstein DR-CVaR constraints.
Using this idea, the next result provides a reformulation of the constraints \eqref{eq:DR_CVAR_individual}.
\begin{thm}
    \label{thm:CVaR_convex_constraints}
    The individual DR-CVaR constraints \eqref{eq:DR_CVAR_individual} are satisfied if the following convex constraints are satisfied.
    \begin{equation}
        \label{eq:convex_DR_CVaR_constraints}
        \begin{aligned}
            \beta_j + \alpha_j^\intercal \hat\mu_k (v) + \tau_{jk} &\sqrt{\alpha_j^\intercal\hat\Sigma_k(L)\alpha_j} + \tilde\varepsilon_k(L) \|\alpha_j\| \leq 0, \\
            &\hspace{1cm} \forall j\in[J], \ \forall k\in[N].
        \end{aligned}
    \end{equation}
    where $\hat\mu_{k} \triangleq \bar{x}_{k}(v) = E_k(\mathcal{A}x_0 + \mathcal{B} v)$ is the propagated mean of the nominal distribution, $\hat\Sigma_k \triangleq \tilde{L}_k\Sigma_{w}\tilde{L}_k^\intercal$ is the propagated covariance of the nominal distribution, $\tilde\varepsilon_k \triangleq \varepsilon(1 + \tau_{jk}^2)^{1/2}\sigma_{\max}^{2}(\tilde{L}_k)$, and
    \begin{equation}
        \label{eq:standard_risk_coefficient}
        \tau \triangleq \sup_{\P\in\mathcal{C}(\mu,\Sigma)} \ \cvar_{1-\gamma}^{\P}\left(\frac{\alpha^\intercal (x - \mu)}{\sqrt{\alpha^\intercal\Sigma\alpha}}\right) = \sqrt{\frac{1 - \gamma}{\gamma}},
    \end{equation}
    is the \textit{standard CVaR risk coefficient}, where $\mathcal{C}(\mu,\Sigma)$ denotes the Chebyshev ambiguity set of all distributions in $\mathcal{S}$ with same mean $\mu$ and covariance $\Sigma$.
\end{thm}
\begin{proof}
    See Appendix~A.
\end{proof}

The constraints in \eqref{eq:convex_DR_CVaR_constraints} are convex, but nonlinear in the decision variable $L$.
However, using Schur complement we can further reformulate these constraints as tractable second-order cone constraints (SOCC) and linear matrix inequalities (LMIs).
\begin{cor}
    \label{cor:CVaR_tractable_constraints}
    The convex constraints \eqref{eq:convex_DR_CVaR_constraints} are equivalent to the following tractable constraints.
    \begin{subequations}~\label{eq:tractable_DR_CVaR_constraints}
        \begin{align}
            &\beta_j + \alpha_j^\intercal\hat\mu_k(v) + \tau_{jk}\|\Sigma_{w}^{1/2}\mathcal{D}^\intercal (I + \mathcal{B}L)^\intercal E_k^\intercal \alpha_{j}\| \nonumber \\
            &\hspace{2cm} + \varepsilon \rho_k \|\alpha_j\| \sqrt{1 + \tau_{jk}^2}\leq 0, \quad \forall j \ \forall k, \\
            &\begin{bmatrix}
                I & E_k(I + \mathcal{B}L)\mathcal{D} \\
                \mathcal{D}^\intercal (I + \mathcal{B}L)^\intercal E_k^\intercal & \rho_k I
            \end{bmatrix} \succeq 0 \quad \forall k,
        \end{align}
    \end{subequations}
    with respect to the decision variables $\{v, L, \rho_k\}$.
\end{cor}
\begin{proof}
    See Appendix~B.
\end{proof}
Along the same lines, in the next section, we reformulate the DR objective function \eqref{eq:DR_cost} as a tractable convex program.
\subsection{DR Objective Reformulation}
Substituting the error dynamics \eqref{eq:error_dynamics} and feedback control $\tilde{\vct u} = K \tilde{\vct x} = L \mathcal{D} \vct w$ into the cost \eqref{eq:DR_cost} yields
\begin{align}
    \mathcal{J} &= \beta \sum_{k=0}^{N-1}\|v_k\| + \max_{\P \in \mathcal{W}}\E_{\P}\left(\tilde{\vct x}^\intercal (\mathcal{Q} + K^\intercal \mathcal{R} K)\tilde{\vct x}\right) \nonumber \\
    &= \beta \sum_{k=0}^{N-1}\|E_k v\| + \max_{\P\in\B_{\varepsilon}^{\|\cdot\|}(\hat\P)}\E_{\P}(\vct w^\intercal \Xi(L) \vct w), \label{eq:DR_cost_w}
\end{align}
where $\mathcal{Q} \triangleq \mathrm{blkdiag}(Q_0,\ldots, Q_{N-1},0)\succeq 0, R \triangleq \mathrm{blkdiag}(R_0,\ldots, R_{N-1})\succ 0$ are the augmented cost matrices, and $\Xi \triangleq \mathcal{D}^\intercal\left((I + \mathcal{B}L)\mathcal{Q}(I + \mathcal{B} L) + L^\intercal\mathcal{R}L\right)\mathcal{D} \succeq 0$.
Thus, we aim to find the worst-case expected value of a quadratic form over the Wasserstein ambiguity set centered around the nominal distribution $\hat\P = \mathcal{N}(0, \Sigma_{w})$.
To this end, it can be shown \cite{DRO_Kuhn_risk} 
that this worst-case expectation is \textit{equivalent} to the worst-case expectation with respect to the associated Gelbrich ambiguity set, provided that the nominal distribution is elliptical.
The following result provides a reformulation of the DR cost.
\begin{thm}
    \label{thm:DR_cost_convex}
    The DR quadratic cost in the objective function \eqref{eq:DR_cost_w} is equivalent to the convex program
    \begin{equation}
    \label{eq:convex_DR_cost}
        \min_{\lambda I \succ \Xi(L)} \ \lambda(\varepsilon^2 - \mathrm{tr}[\Sigma_{w}] + \lambda \mathrm{tr}[\Sigma_{w} (\lambda I - \Xi(L))^{-1}]).
    \end{equation}
\end{thm}
\begin{proof}
    See Appendix~C.
\end{proof}

Similar to the DR-CVaR constraints \eqref{eq:convex_DR_CVaR_constraints}, the reformulated DR cost \eqref{eq:convex_DR_cost} is convex but nonlinear in the decision variables $\gamma$ and $L$.
Using Schur complements, we can reformulate \eqref{eq:convex_DR_cost} as the following SDP.
\begin{cor}
    \label{cor:DR_cost_tractable}
    The convex program \eqref{eq:convex_DR_cost} is equivalent to the semi-definite program
    \begin{subequations}~\label{eq:tractable_DR_cost}
        \begin{align}
            &\min_{\substack{\lambda \geq 0 \\ \Gamma,\Psi \succeq 0}} \ &&\lambda(\varepsilon^2 - \mathrm{tr}[\Sigma_{w}]) + \mathrm{tr}[\Gamma], \label{eq:DR_cost_obj} \\
            &\hspace{0.5cm}\mathrm{s.t.} 
            &&\begin{bmatrix}
                \Gamma & \lambda \Sigma_{w}^{1/2} \\
                \lambda \Sigma_{w}^{1/2} & \Psi
            \end{bmatrix} \succeq 0, \label{eq:DR_cost_constraint1} \\
            &&&\begin{bmatrix}
                \lambda I - \mathcal{D}^\intercal \tilde{M}(L)\mathcal{D} - \Psi & \mathcal{D}^\intercal L^\intercal \\
                L\mathcal{D} & \tilde{\mathcal{R}}^{-1}
            \end{bmatrix} \succeq 0, \label{eq:DR_cost_constraint2}
        \end{align}
    \end{subequations}
    where $\tilde{M} \triangleq \mathcal{Q} + M(L) + M^\intercal(L), \ M \triangleq \mathcal{Q}\mathcal{B}L$, and $\tilde{\mathcal{R}}\triangleq \mathcal{B}^\intercal\mathcal{Q}\mathcal{B} + \mathcal{R}$.
\end{cor}
\begin{proof}
    See Appendix~D.
\end{proof}
Lastly, in the next section, we reformulate the terminal constraints that require the terminal distributional uncertainty of the state to lie within a desired target ambiguity set $\S_{f}$.
\subsection{Terminal Constraints}
Using \eqref{eq:dist_uncertainty_state}, the terminal propagated ambiguity set of the state is given by
\begin{equation}
    \label{eq:terminal_ambiguity_state}
    \S_{N} = \B_{\varepsilon}^{\|\cdot\|\circ\tilde{L}_{N}^\dagger}(\hat\P_{N}), \quad \hat\P_{N} = \mathcal{N}(\bar{x}_{N}(v), \tilde{L}_{N}\Sigma_{w}\tilde{L}_{N}^\intercal).
\end{equation}
To ensure the inclusion $\S_{N} \subseteq \B_{\delta}^{\|\cdot\|}(\hat\P_f)$, we first note that it can be shown \cite{capture_propagate_control_OT} that
\begin{equation}
    \label{eq:ambiguity_set_relaxation}
    \B_{\varepsilon}^{\|\cdot\|\circ\tilde{L}^\dagger}(\hat\P) \subseteq \B_{\varepsilon}^{\sigma_{\min}^2(\tilde{L})\|\cdot\|}(\hat\P) = \B_{\varepsilon\sigma_{\max}^{2}(\tilde{L})}^{\|\cdot\|}(\hat\P).
\end{equation}
Thus, since both $\hat\P_{N}$ and $\hat\P_f$ are normally distributed, it is sufficient to enforce the constraints
\begin{equation}
    \label{eq:terminal_constraints}
    \bar{x}_{N}(v) = \mu_{f}, \quad \Sigma_{x_{N}}(L) \preceq \Sigma_{f}, \quad \varepsilon \sigma_{\max}^{2}(\tilde{L}_{N}) \leq \delta.
\end{equation}
\begin{rem}
    The first two constraints in \eqref{eq:terminal_constraints} are equivalent to those in the traditional CS literature \cite{Max2}.
    Indeed, the main goal of covariance control is to steer the covariance (and mean) of the state distribution to some desired terminal covariance, where the relaxation $\Sigma_{x_{N}} \preceq \Sigma_{f}$ is often introduced to make the terminal constraints tractable.
    In this context, however, the first two constraints align the center distributions of the terminal state, while the extra constraint in \eqref{eq:terminal_constraints} can be interpreted as a way to robustify against distributionally uncertainty in the terminal state, providing an extra layer of safety guarantee against unknown disturbances.
\end{rem}

To this end, the terminal mean constraint in \eqref{eq:terminal_constraints} is simply a linear constraint in $v$, given by
\begin{equation}
    \label{eq:terminal_mean_constraint}
    E_{N}(\mathcal{A} x_0 + \mathcal{B} v) - \mu_{f} = 0.
\end{equation}
Second, the terminal covariance constraint in \eqref{eq:terminal_constraints} can be written as the following LMI
\begin{align}
    \begin{bmatrix}
        \Sigma_{f} & E_{N}(I + \mathcal{B}L)\mathcal{D}\Sigma_{w}^{1/2} \\
        \Sigma_{w}^{1/2}\mathcal{D}^\intercal (I + \mathcal{B}L)^\intercal E_{N}^\intercal & I
    \end{bmatrix} \succeq 0. \label{eq:terminal_cov_constraint}
\end{align}
Lastly, noting that $\sigma_{\max}(L) = \|L\|$ and using the Schur complement, the terminal distributional uncertainty constraint can be written as the LMI
\begin{equation}
    \label{eq:terminal_DR_constraint}
    \begin{bmatrix}
        I & E_{N}(I + \mathcal{B}L)\mathcal{D} \\
        \mathcal{D}^\intercal (I + \mathcal{B} L)^\intercal E_{N}^\intercal & (\delta/\varepsilon) I
    \end{bmatrix} \succeq 0.
\end{equation}
In summary, combining all ingredients of Sections~IV.A-IV.C, the DR-DS problem can be solved as the SDP
\begin{align*}
    &\min_{\substack{v, K \\ \rho_k, \lambda \geq 0 \\ \Gamma, \Psi \succeq 0}} \ &&\beta\sum_{k=0}^{N-1} \|E_k v\| + \lambda(\varepsilon^2 - \mathrm{tr}[\Sigma_{w}]) + \mathrm{tr}[\Gamma] \\
    &\quad \mathrm{s.t.} &&\eqref{eq:tractable_DR_CVaR_constraints}, \eqref{eq:DR_cost_constraint1}, \eqref{eq:DR_cost_constraint2}, \eqref{eq:terminal_mean_constraint},\eqref{eq:terminal_cov_constraint},\eqref{eq:terminal_DR_constraint}.
\end{align*}

\section{NUMERICAL EXAMPLES}
\subsection{Double Integrator Path Planning}
As a first example to showcase the proposed DR-DS framework, consider a 2D double integrator integrator with dynamics
\begin{equation*}
    A = 
    \begin{bmatrix}
        I_{2} & \Delta T I_{2} \\ 
        0_{2} & I_{2}
    \end{bmatrix}, \quad B = 
    \begin{bmatrix}
        \frac{\Delta t^2}{2} I_{2} \\
        \Delta t I_{2}
    \end{bmatrix}, \quad D = 5\times 10^{-3} I_{4},
\end{equation*}
with initial state $x_0 = [-1, 2, 0.1, -0.1]^\intercal$, and i.i.d. nominal disturbances drawn from $\hat{\mathbb{P}}_{w} = \mathcal{N}(0, I_{4})$.
The nominal target state distribution is $\hat{\mathbb{P}}_f = \mathcal{N}(0, (0.1/3)^2 I_{4})$, and the desired target ambiguity set has radius $\delta= 0.05$.
Lastly, the planning horizon has $N = 20$ time steps, $\Delta t = 0.3$, and we enforce probabilistic constraints with respect to the polytope defined by $\alpha_{1, [8 : N]} = [-1, 0, 0, 0]^\intercal, \ \alpha_{2, [8 : N]} = [1, 0, 0, 0]^\intercal$, and $b_{1, [8 : N]} = b_{2, [8 : N]} = -0.2$, which probabilistically enforces $|x| \leq 0.2$ in the terminal stage of planning, with probability $\gamma_{jk} = 0.05$ along each individual constraint.
We compare the performance of the DR-DS control to that of the baseline CS solution with chance constraints \cite{Max2}.
The convex programs were all solved using the YALMIP optimization suite \cite{YALMIP} with the MOSEK solver \cite{MOSEK}.

Firstly, we compare the optimal solutions subject to the nominal disturbances in Figure~\ref{fig:double_integrator_nominal}.
\begin{figure}[!htb]
    \centering
    \includegraphics[width=\linewidth, trim={1cm 0.5cm 1cm 1.5cm}, clip]{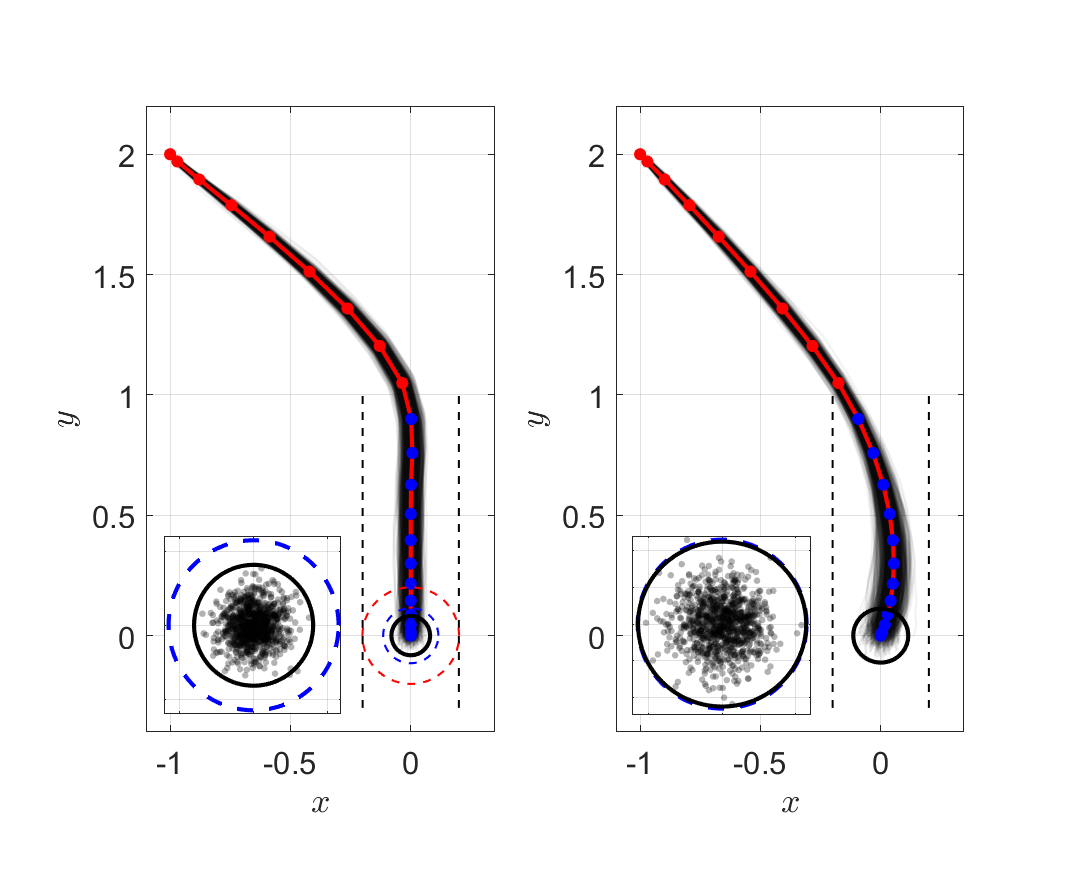}
    \caption{Optimal trajectories for (left) DR-DS solution with $\varepsilon = 15$, and (right) baseline CS solution, subject to nominal disturbance $\mathbb{P}_w$.}
    \label{fig:double_integrator_nominal}
\end{figure}
Clearly, in the nominal case, when the disturbance is well-understood, both solutions are able to successfully steer to the desired terminal distribution and satisfy the constraints, since by construction, this is what CS is designed to do.
Additionally, the empirical risk for the DR-CVaR constraints and (conservative) chance-constraints are both zero for 1,000 Monte-Carlo samples.
Interestingly, however, note that the DR-DS solution steers to a smaller terminal covariance compared to that of CS.
For reference, the red covariance ellipse in the left plot in Figure~\ref{fig:double_integrator_nominal} is the \textit{maximal} normal distribution in the target terminal ambiguity set with respect to the underlying structure of zero-mean Gaussian's, which is computed from $\W(\Sigma_f, \eta_f^2\Sigma_f) = \delta$, or equivalently, by using the Wasserstein distance between two normal distributions, as
\begin{equation}
    \label{eq:maximal_covariance}
    \eta_{f} = 1 + \frac{\delta}{\sqrt{\mathrm{tr}(\Sigma_f)}}.
\end{equation}
Thus, the DR-DS framework will steer the state distribution to $\Sigma_{N} \preceq \eta_f^2 \Sigma_{f}$ for any disturbance $\mathbb{P}_{w} \in \mathbb{B}_{\varepsilon}^{\|\cdot\|}(\hat{\mathbb{P}}_{w})$.

To illustrate this, Figure~\ref{fig:double_integrator_maximal} shows the performance of the two methods when the noise distribution is now given by $\P_{w} = \mathcal{N}(0, \eta_{w}^2 I)$, where, as in \eqref{eq:maximal_covariance}, $\eta_{w} = 1 + \varepsilon / \sqrt{\mathrm{tr}(\Sigma_{w})}$ is the maximal covariance in the ambiguity set $\mathbb{B}_{\varepsilon}^{\|\cdot\|}(\hat{\P}_{w})$.
\begin{figure}[!htb]
    \centering
    \includegraphics[width=\linewidth, trim={1cm 0.5cm 1cm 1.5cm}, clip]{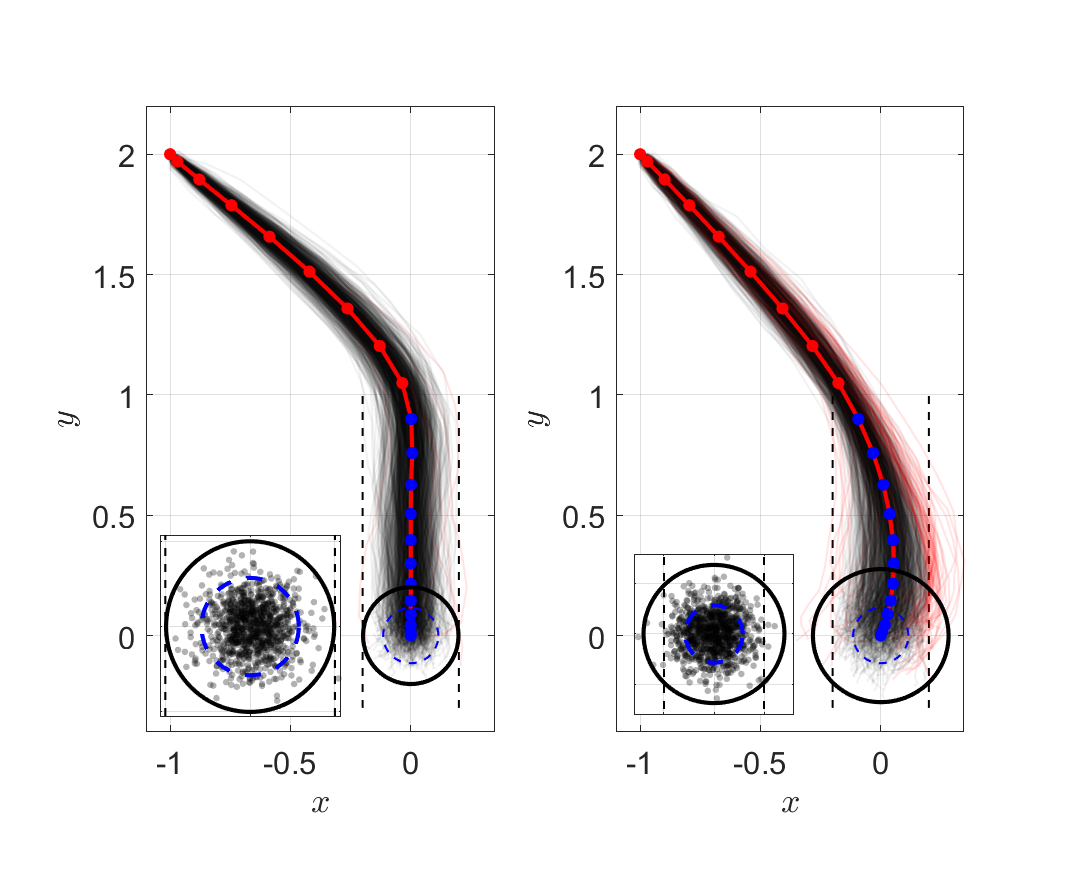}
    \caption{Optimal trajectories for (left) DR-DS solution with $\varepsilon = 15$, and (right) baseline CS solution, subject to \textit{maximal} disturbance $\mathbb{P}_w$ in disturbance ambiguity set $\B_{\varepsilon}^{\|\cdot\|}(\hat{\P}_{w})$.
    }
    \label{fig:double_integrator_maximal}
\end{figure}
When the true noise affecting the system is mis-characterized and not equivalent to the noise the system was designed to handle, the baseline CS is unable to steer to the terminal covariance nor satisfy the chance-constraints with the desired level of risk.
The DR-DS solution, on the other hand, is agnostic to the noise distribution by design (within limits, of course), and is able to steer the state distribution to the terminal ambiguity set and still satisfies the CVaR constraints.
Indeed, the empirical risk of constraint violation is 0.1\% and 5.5\%, respectively, for DR-DS and baseline CS.

Lastly, we would also like to see the effect of \textit{non}-Gaussian disturbances acting on the system.
Recall that since we assumed that the structural ambiguity set is the entire probability space, this implies that the DR-DS should be able to account for any disturbance distribution $\P\in\mathcal{P}_{2}(\R^{d})$ such that $\W(\P, \hat\P_{w}) \leq \varepsilon$.
To this end, we inject noise from a $t$-distribution with 3 DOF, and the resulting optimal trajectories are shown in Figure~\ref{fig:double_integrator_non_gaussian}.
\begin{figure}[!htb]
    \centering
    \includegraphics[width=\linewidth, trim={1cm 0.5cm 1cm 1cm}, clip]{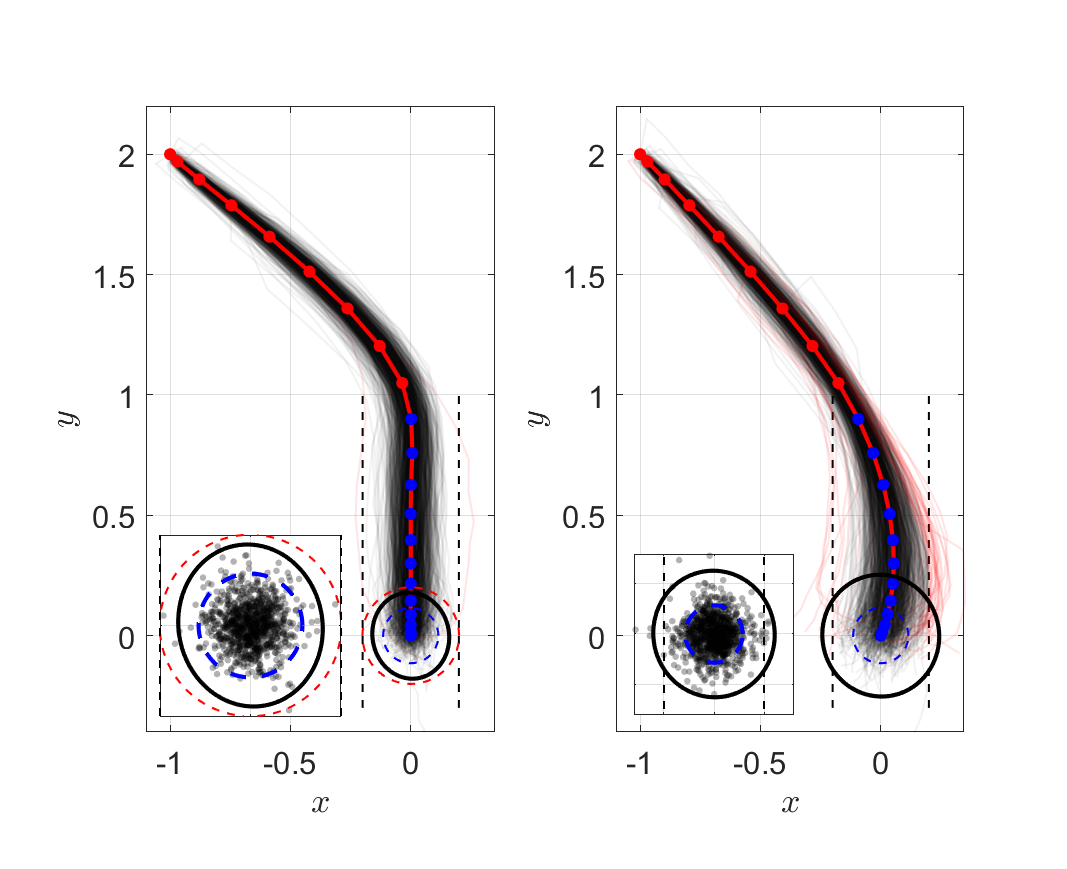}
    \caption{Optimal trajectories for (left) DR-DS solution with $\varepsilon = 15$, and (right) baseline CS solution, subject to \textit{non-Gaussian} t-distribution disturbance.
    }
    \label{fig:double_integrator_non_gaussian}
\end{figure}
Notably, the feedback gains in DR-DS are able to shape the covariance ellipses to satisfy the desired terminal ambiguity set constraints, while the baseline CS fails to take into account this non-Gaussian structure.
The DR-DS solution is also still able to satisfy path constraints with a joint risk of 0.3\%, while the heavy-tail nature of the disturbances skews the transient dispersion of the states for the baseline CS solution, resulting in a joint risk of 3.5\%.
Thus, by incorporating distributional robustness both into the constraints of the system, as well as to the terminal ingredients, we are able to steer a much broader class of systems, whose solutions are robust to uncertainties in our knowledge of the disturbance structure.
\subsection{Quadrotor Landing with Wind Turbulence}
We now turn our attention to a more practical setting of landing a UAV in the presence of harsh wind turbulence.
To this end, we model the quadrotor as a 9-DOF system governed by the nonlinear dynamics
\begin{equation}
    \label{eq:drone_dynamics}
    \begin{aligned}
        \dot{r} &= v, \\
        \dot{q} &= S(q)\omega, \\ 
        \dot{v} &= \frac{1}{m}(-e_3 g + R(q) \hat{e}_3 \tau),
    \end{aligned}
\end{equation}
where $r, v$ represent the position and velocity in an inertial frame, $q \triangleq [\phi, \theta, \psi]^\intercal$ represents the attitude parametrized by ZYX Euler angles, $S(q)$ denotes the rotation matrix for the angular rates from body frame to inertial frame, and $R(q)$ denotes the standard ZYX rotation matrix.
Additionally, the control inputs are the body frame angular rates $\omega$ and the net vertical acceleration $\tau$.
Lastly, $e_3$ and $\hat e_3$ denote the unit vectors along the $z$-axis in the inertial and body frame, respectively.
The remaining parameter values may be found in Table~\ref{tab:drone_parameters}.

We first compute a reference trajectory and control by solving an optimal control problem for the nonlinear system \eqref{eq:drone_dynamics} with initial state $x_0 = [-5, 3, 10, 0_{1\times 6}]^\intercal$ using CasADi \cite{CasADi}.
We then linearize the system around this reference, and subsequently discretize it with $T = 5$ seconds time horizon and $N = 10$ time steps.
Specific details on the exact procedure performed may be found in \cite{JoshJack} for reference.
For the disturbance model, we use the Dryden wind turbulence model \cite{Dryden_model}, which is a zero-mean, stationary Gaussian process defined by its power spectral density (PSD) $\Phi$.
Specifically, we assume six turbulence channels for the three linear and angular velocities respectively.
Since the state vector only contains the attitude, we assume the disturbances enter as $w_{q_k} = \Delta t w_{\dot{q}_k}$, and subsequently $D = [0_{3\times 6}; I_{6}]$.
To compute the covariance matrix $\Sigma_{w}\in\R^{Nd\times Nd}$ of the turbulence for each channel $i\in[d]$, note that by the Wiener-Khintchine theorem \cite{Wiener_theorem}, the covariance function is the inverse Fourier transform of the PSD, that is, 
\begin{equation}
    \Sigma_{i}(\tau) = \int_{\R} \Phi_i(\omega) e^{2\pi \mathrm{i}\omega\tau} \ \d\tau.
\end{equation}
\renewcommand{\arraystretch}{1.3}
\begin{table*}[!htb]
    \begin{minipage}{\columnwidth}
        \centering
        \begin{tabular}{lcccrlc}
            \toprule 
            \midrule
            & \multicolumn{3}{c}{$3\sigma$} & \phantom{a} & \multicolumn{2}{c}{System}\\ 
            \cmidrule{2-4} \cmidrule{6-7}
            & $x$ & $y$ & $z$\\ 
            \midrule 
            $r_f$ (m) & 1.5 & 1.5 & 0.15 && $m$ (kg) & 0.8 \\ 
            $q_f$ (deg) & 5 & 5 & 10 && $g$ (m/s$^2$) & 9.81 \\ 
            $v_f$ (m/s) & 0.5 & 0.5 & 0.05 && $V_0$ (m/s) & 1\\
            \midrule
            \bottomrule 
        \end{tabular} 
        \caption{Parameter values for drone landing problem.}
        \label{tab:drone_parameters}
    \end{minipage}
    \hfill
    \begin{minipage}{\columnwidth}
        \begin{tabular}{lcllllll}
            \toprule
            \multicolumn{2}{c}{$V_0$ (m/s)} & 1 & 5 & 10 & 15 & 20 & 50 \\
            \midrule
            \multirow{3}{*}{$(\sigma/\bar{\sigma})_{r_{N}}$} & $x$ & 2.967 & 2.369 & 2.935 & 3.899 & 5.123 & 14.095 \\ 
            & $y$ & 3.687 & 4.296 & 5.376 & 6.199 & 6.863 & 9.218 \\ 
            & $z$ & 1.072 & 1.067 & 1.087 & 1.105 & 1.121 & 1.171 \\
            \bottomrule
        \end{tabular}
        \caption{Ratio of terminal position standard deviations between DR-DS and baseline CS for different mean wind speeds.}
        \label{tab:terminal_std_ratios}
    \end{minipage}
\end{table*}
We choose an noise ambiguity set radius $\varepsilon = 1$ as well as a desired terminal state radius $\delta = 0.1$.
For the terminal reference distribution, the state elements should have a $3\sigma$ value no greater than those in Table~\ref{tab:drone_parameters}.
For the nominal turbulence distribution, we assume a mean wind speed of $V_0$.
As in the first example, both the baseline CS and DR-DS solutions successfully steer the vehicle under the nominal calm turbulence model, thus we leave this out for brevity.

As in the first example, we first present the nominal optimal trajectories subject to the reference noise distribution.
Figure~\ref{fig:terminal_drone_nominal} shows the terminal splashpoints from the result of 1,000 Monte Carlo trials, as well as the covariances along each $i-j$ plane.
We see that under a calm turbulence, all constraints are met quite conservatively.
\begin{figure}[!htb]
    \begin{subfigure}[b]{0.23\textwidth}
        \centering
        \includegraphics[width=\linewidth, trim={0cm, 0cm, 0cm, 0.6cm}, clip]{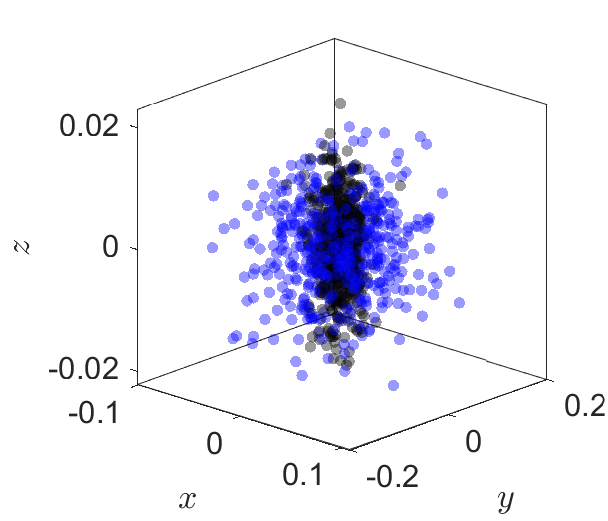}
        \caption{Terminal splash-points of MC trajectories.}
    \end{subfigure}
    \quad 
    \begin{subfigure}[b]{0.23\textwidth}
        \centering
        \includegraphics[width=\linewidth, trim={0cm, 0cm, 0.8cm, 0cm}, clip]{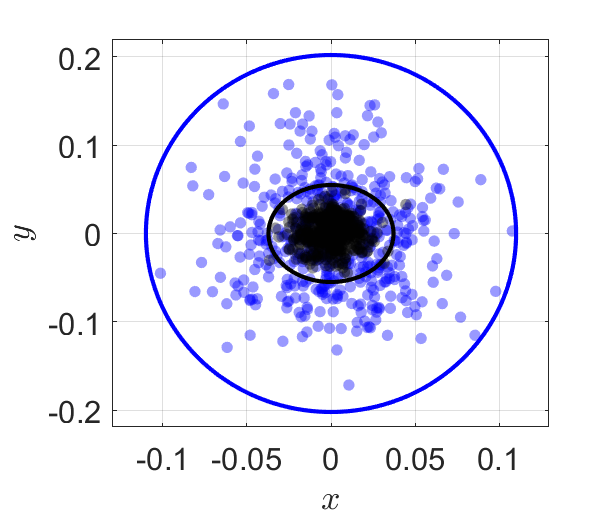}
        \caption{Terminal covariance of $x-y$ position.}
    \end{subfigure}
    \\ 
    \begin{subfigure}[b]{0.23\textwidth}
        \centering
        \includegraphics[width=\linewidth, trim={0cm, 0cm, 0.8cm, 0cm}, clip]{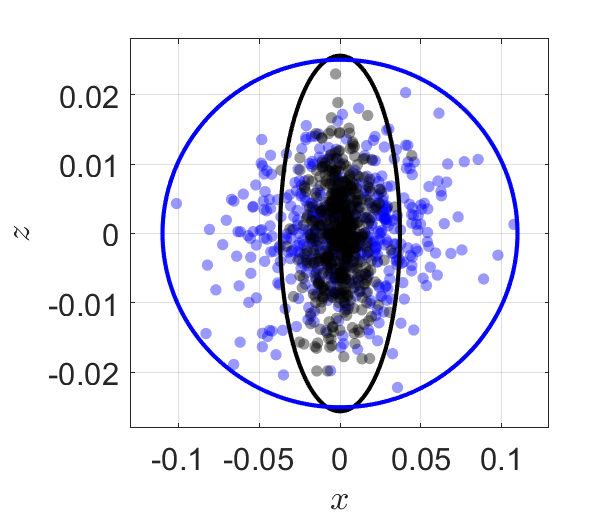}
        \caption{Terminal covariances of $x-z$ position.}
    \end{subfigure}
    \quad 
    \begin{subfigure}[b]{0.23\textwidth}
        \centering
        \includegraphics[width=\linewidth, trim={0cm, 0cm, 0.8cm, 0cm}, clip]{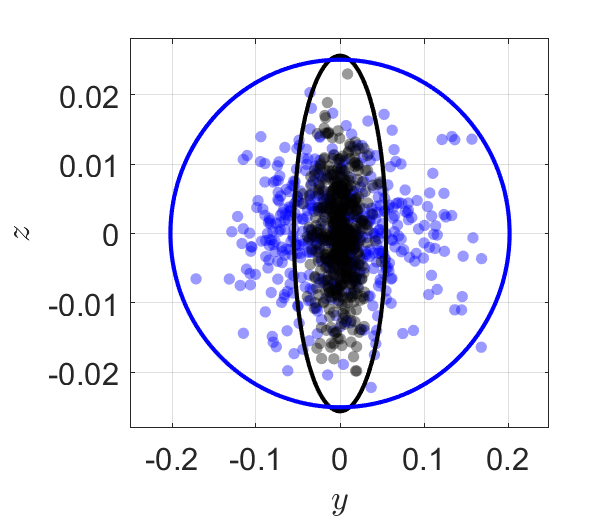}
        \caption{Terminal covariances of $y-z$ position.}
    \end{subfigure}
    \caption{Terminal position distribution and propagated Monte Carlo samples for DR-DS (black) and baseline CS (blue) under nominal turbulence model.}
    \label{fig:terminal_drone_nominal}
\end{figure}
Similar to the first example, we see that even in the nominal case, the DR-DS solution achieves a smaller covariance compared to that of the CS solution.
Intuitively, this suggests that distributional robustness against a set of distributions implies more conservative nominal solutions.

Next, we inject a severe disturbance into the system dynamics, namely with a mean wind speed $V_0 = 20$ m/s.
\begin{figure}[!htb]
    \begin{subfigure}{.5\textwidth}
        \centering
        \includegraphics[width=.78\linewidth]{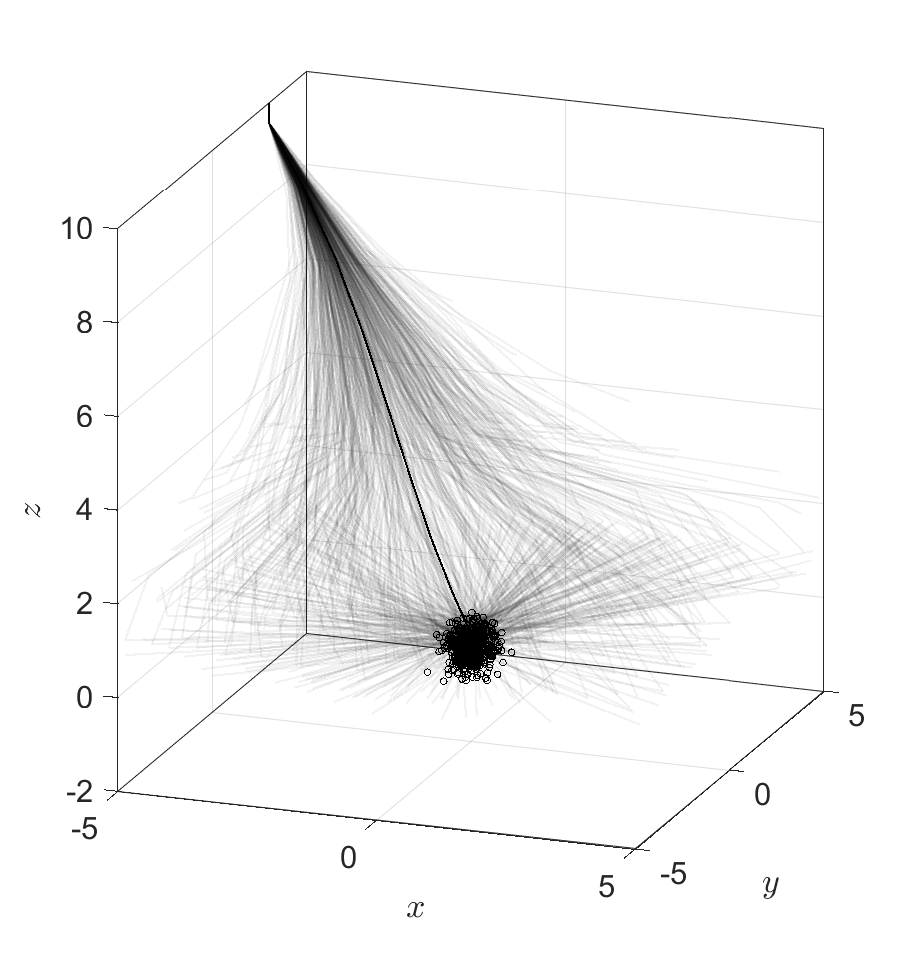}
        \caption{DR-DS position trajectories.}
    \end{subfigure}
    \begin{subfigure}{.5\textwidth}
        \centering
        \includegraphics[width=.78\linewidth]{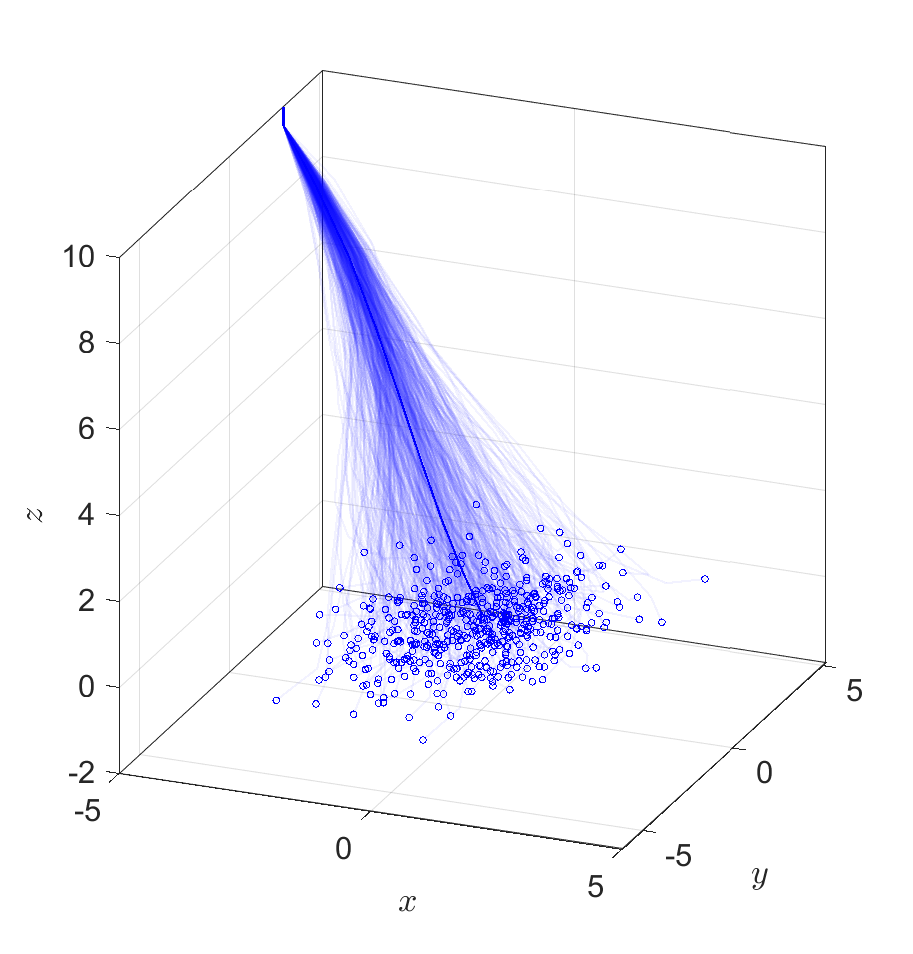}
        \caption{CS position trajectories.}
    \end{subfigure}
    \caption{Monte Carlo trajectories and terminal splashpoints for (a) DR-DS and (b) CS solutions with severe wind turbulence.}
    \label{fig:drone_trajs_severe}
\end{figure}
The trajectories along with the terminal splash-points are shown in Figure~\ref{fig:drone_trajs_severe}.
As expected, the DR-DS solution exhibits much more distributional control of the terminal state, even at large disturbances, though at the expense of wider dispersion in the transient motion.
One could potentially reduce these adverse affects through imposing DR-CVaR constraints on the path of the quadrotor, however we leave this investigation to future work.
Lastly, we would like to quantitatively determine the scale of robustness for increasing levels of wind turbulence.
Table~\ref{tab:terminal_std_ratios} displays the ratios of the terminal position standard deviations of the DR-DS solution $\sigma_{r_N}$ with that of the CS solution $\bar\sigma_{r_N}$.
As mentioned, even in the nominal case, we get a reduction in the covariances, however as the turbulence increases, this reduction grows substantially, giving an almost 15x reduction in standard deviation at extreme turbulence levels.
The variance in the $z$-position, however, does not reduce by all that much, and this is most likely due to the fact that the wind turbulence in the lateral and longitudinal directions is more pronounced than that of the vertical direction, especially at larger wind speeds.

\section{CONCLUSION}
In this work, we have developed a distributionally-robust density control method for steering the distributional uncertainty of the state of a linear dynamical system subject to imperfect knowledge of the disturbances affecting the system.
Through characterizing the distributional uncertainty in the noise distribution via Wasserstein ambiguity sets, we are able to propagate the ambiguity set of the state through the LTI dynamics, and tractably formulate the DR objective function, DR-CVaR constraints, and terminal ambiguity set constraints as an SDP, which can be solved in polynomial time.
We showcased the proposed methodology on both a double integrator steering problem and a drone landing problem, illustrating safe planning under not only mis-characterized i.i.d. Gaussian disturbances, but also imprecise GP turbulence modeling and heavy-tailed distributional robustness.
Future work will aim to investigate tractable formulations of the DR-DS problem in data-driven settings, where the reference noise distribution is constructed from empirical samples.

\section{ACKNOWLEDGMENTS}

This work has been supported by NASA University Leadership Initiative award 80NSSC20M0163 and ONR award N00014-18-1-2828. The article solely reflects the opinions and conclusions of its authors and not any NASA entity.
We would also like to thank Dr. Daniel Kuhn for his discussion and helpful insights.

\bibliographystyle{IEEEtran}
\bibliography{refs.bib}

\begin{thebibliography}{10}
\providecommand{\url}[1]{#1}
\csname url@samestyle\endcsname
\providecommand{\newblock}{\relax}
\providecommand{\bibinfo}[2]{#2}
\providecommand{\BIBentrySTDinterwordspacing}{\spaceskip=0pt\relax}
\providecommand{\BIBentryALTinterwordstretchfactor}{4}
\providecommand{\BIBentryALTinterwordspacing}{\spaceskip=\fontdimen2\font plus
\BIBentryALTinterwordstretchfactor\fontdimen3\font minus \fontdimen4\font\relax}
\providecommand{\BIBforeignlanguage}[2]{{%
\expandafter\ifx\csname l@#1\endcsname\relax
\typeout{** WARNING: IEEEtran.bst: No hyphenation pattern has been}%
\typeout{** loaded for the language `#1'. Using the pattern for}%
\typeout{** the default language instead.}%
\else
\language=\csname l@#1\endcsname
\fi
#2}}
\providecommand{\BIBdecl}{\relax}
\BIBdecl

\bibitem{JP_DR}
V.~Renganathan, J.~Pilipovsky, and P.~Tsiotras, ``Distributionally robust covariance steering with optimal risk allocation,'' in \emph{American Control Conference}, 2023, pp. 2607--2614.

\bibitem{HS2}
A.~F. Hotz and R.~E. Skelton, ``Covariance control theory,'' \emph{International Journal of Control}, vol.~46, no.~1, pp. 13--32, 1987.

\bibitem{EB1}
E.~Bakolas, ``Finite-horizon separation-based covariance control for discrete-time stochastic linear systems,'' in \emph{57th IEEE Conference on Decision and Control}, Miami Beach, FL, Dec 17--19, 2018, pp. 3299--3304.

\bibitem{Max1}
M.~Goldshtein and P.~Tsiotras, ``Finite-horizon covariance control of linear time-varying systems,'' in \emph{56th IEEE Conference on Decision and Control}, Melbourne, Australia, Dec 12--15 2017, pp. 3606--3611.

\bibitem{EB2}
E.~Bakolas, ``Optimal covariance control for discrete-time stochastic linear systems subject to constraints,'' in \emph{55th IEEE Conference on Decision and Control}, Las Vegas, NV, Dec 12--14, 2016, pp. 1153--1158.

\bibitem{Max2}
K.~Okamoto, M.~Goldshtein, and P.~Tsiotras, ``Optimal covariance control for stochastic systems under chance constraints,'' \emph{IEEE Control System Letters}, vol.~2, pp. 266--271, 2018.

\bibitem{exact_CS_2}
G.~Rapakoulias and P.~Tsiotras, ``Discrete-time optimal covariance steering via semidefinite programming,'' in \emph{62nd IEEE Conference on Decision and Control}, Marina Bay Sands, Singapore, Dec. 13--15, 2023.

\bibitem{OFCS}
J.~Ridderhof, K.~Okamoto, and P.~Tsiotras, ``Chance constrained covariance control for linear stochastic systems with output feedback,'' in \emph{59th IEEE Conference on Decision and Control}, Jeju, Korea (South), 2020, pp. 1758--1763.

\bibitem{Pili8}
J.~Pilipovsky and P.~Tsiotras, ``Computationally efficient chance constrained covariance control with output feedback,'' 2023, arXiv:2310.02485.

\bibitem{JP_DDCS}
------, ``Data-driven covariance steering control design,'' in \emph{62nd IEEE Conference on Decision and Control}, Marina Bay Sands, Singapore, 2023, pp. 2610--2615.

\bibitem{Pili4}
------, ``Covariance steering with optimal risk allocation,'' \emph{IEEE Transactions on Aerospace and Electronic Systems}, vol.~57, no.~6, pp. 3719--3733, 2021.

\bibitem{PDG_Jack}
J.~Ridderhof and P.~Tsiotras, \emph{Minimum-fuel Powered Descent in the Presence of Random Disturbances}.

\bibitem{JoshJack}
J.~Ridderhof, J.~{Pilipovsky}, and P.~{Tsiotras}, ``Chance-constrained covariance control for low-thrust minimum-fuel trajectory optimization,'' in \emph{AAS/AIAA Astrodynamics Specialist Conference}, Lake Tahoe, CA, Aug 9--13 2020.

\bibitem{spacecraft_Oguri}
K.~Oguri and J.~W. McMahon, ``Robust spacecraft guidance around small bodies under uncertainty: Stochastic optimal control approach,'' \emph{Journal of Guidance, Control, and Dynamics}, vol.~44, no.~7, pp. 1295--1313, 2021.

\bibitem{spacecraft_Benedikter}
B.~Benedikter, A.~Zavoli, Z.~Wang, S.~Pizzurro, and E.~Cavallini, ``Convex approach to covariance control with application to stochastic low-thrust trajectory optimization,'' \emph{Journal of Guidance, Control, and Dynamics}, vol.~45, no.~11, pp. 2061--2075, 2022.

\bibitem{CSMPC_driving}
J.~Knaup, K.~Okamoto, and P.~Tsiotras, ``Safe high-performance autonomous off-road driving using covariance steering stochastic model predictive control,'' \emph{IEEE Transactions on Control Systems Technology}, vol.~31, no.~5, pp. 2066--2081, 2023.

\bibitem{CS_GRF}
J.~Ridderhof and P.~Tsiotras, ``Chance-constrained covariance steering in a gaussian random field via successive convex programming,'' \emph{Journal of Guidance, Control, and Dynamics}, vol.~45, no.~4, pp. 599--610, 2022.

\bibitem{CS_martingale}
F.~Liu and P.~Tsiotras, ``Optimal covariance steering for continuous-time linear stochastic systems with martingale additive noise,'' \emph{IEEE Transactions on Automatic Control}, pp. 1--8, 2023.

\bibitem{CS_multiplicative}
I.~M. Balci and E.~Bakolas, ``Covariance steering of discrete-time linear systems with mixed multiplicative and additive noise,'' in \emph{American Control Conference}, San Diego, CA, 2023, pp. 2586--2591.

\bibitem{DRO_Kuhn_risk}
V.~A. Nguyen, S.~Shafiee, D.~Filipović, and D.~Kuhn, ``Mean-covariance robust risk measurement,'' 2023.

\bibitem{DRO_Kuhn_empirical}
P.~M. Esfahani and D.~Kuhn, ``Data-driven distributionally robust optimization using the wasserstein metric: performance guarantees and tractable reformulations,'' \emph{Mathematical Programming}, vol. 171, pp. 115 -- 166, 2015.

\bibitem{aolaritei2023distributional}
L.~Aolaritei, N.~Lanzetti, H.~Chen, and F.~Dörfler, ``Distributional uncertainty propagation via optimal transport,'' 2023.

\bibitem{capture_propagate_control_OT}
L.~Aolaritei, N.~Lanzetti, and F.~Dörfler, ``Capture, propagate, and control distributional uncertainty,'' in \emph{62nd IEEE Conference on Decision and Control}, Marina Bay Sands, Singapore, 2023, pp. 3081--3086.

\bibitem{DR_DeePC}
J.~Coulson, J.~Lygeros, and F.~Dörfler, ``Distributionally robust chance constrained data-enabled predictive control,'' \emph{IEEE Transactions on Automatic Control}, vol.~67, no.~7, pp. 3289--3304, 2022.

\bibitem{Wasserstein_convergence}
N.~Fournier and A.~Guillin, ``{On the rate of convergence in Wasserstein distance of the empirical measure},'' \emph{{Probability Theory and Related Fields}}, vol. 162, no. 3-4, p. 707, Aug 2015.

\bibitem{Ono_chance_constraints}
L.~Blackmore, M.~Ono, and B.~C. Williams, ``Chance-constrained optimal path planning with obstacles,'' \emph{IEEE Transactions on Robotics}, vol.~27, no.~6, pp. 1080--1094, 2011.

\bibitem{tractable_CC}
A.~Nemirovski, ``On safe tractable approximations of chance constraints,'' \emph{European Journal of Operational Research}, vol. 219, no.~3, pp. 707--718, 2012, feature Clusters.

\bibitem{DRO_Kuhn_ML}
D.~Kuhn, P.~M. Esfahani, V.~A. Nguyen, and S.~Shafieezadeh-Abadeh, ``Wasserstein distributionally robust optimization: Theory and applications in machine learning,'' 2019.

\bibitem{YALMIP}
J.~Lofberg, ``{YALMIP}: A toolbox for modeling and optimization in matlab,'' in \emph{IEEE International Symposium on Computer Aided Control Systems Design}, Taipei, Taiwan, 2004, pp. 284--289.

\bibitem{MOSEK}
MOSEK ApS, The MOSEK Optimization Toolbox for MATLAB Manual. Version 8.1., 2017. [Online]. Available: http://docs.mosek.com.

\bibitem{CasADi}
J.~A.~E. Andersson, J.~Gillis, G.~Horn, J.~B. Rawlings, and M.~Diehl, ``{CasADi} -- {A} software framework for nonlinear optimization and optimal control,'' \emph{Mathematical Programming Computation}, vol.~11, no.~1, pp. 1--36, 2019.

\bibitem{Dryden_model}
``Flying qualities of piloted aircraft,'' {M}IL-STD-1797A, February, 1990.

\bibitem{Wiener_theorem}
N.~Wiener, \emph{{Extrapolation, Interpolation, and Smoothing of Stationary Time Series: With Engineering Applications}}.\hskip 1em plus 0.5em minus 0.4em\relax The MIT Press, 08 1949.

\bibitem{Hespanha09}
J.~P. Hespanha, \emph{Linear Systems Theory}.\hskip 1em plus 0.5em minus 0.4em\relax Princeton, New Jersey: Princeton Press, Feb. 2018, iSBN13: 9780691179575.

\bibitem{DRO_Kuhn_MSE}
V.~A. Nguyen, S.~Shafieezadeh-Abadeh, D.~Kuhn, and P.~Mohajerin~Esfahani, ``Bridging bayesian and minimax mean square error estimation via wasserstein distributionally robust optimization,'' \emph{Mathematics of Operations Research}, vol.~48, no.~1, pp. 1--37, 2023.

\bibitem{Sion1958OnGM}
M.~Sion, ``On general minimax theorems,'' \emph{Pacific Journal of Mathematics}, vol.~8, pp. 171--176, 1958.

\end{thebibliography}

\appendix
\section*{A.~Proof of Theorem~\ref{thm:CVaR_convex_constraints}}
\setcounter{equation}{0}
\renewcommand{\theequation}{A.\arabic{equation}}
For ease of notation, we drop the subscripts from all variables, i.e., $\alpha_j = \alpha, \gamma_{jk} = \gamma$, and so on.
First, note the following sequence of inclusions of ambiguity sets
\begin{equation*}
    \B_{\varepsilon}^{\|\cdot\|\cdot \tilde{L}_k^\dagger}(\hat\P_k) \subseteq \B_{\varepsilon\sigma_{\max}^{2}(\tilde{L}_k)} \subseteq \mathcal{G}_{\varepsilon\sigma_{\max}^{2}(\tilde{L}_k)},
\end{equation*}
where the first inclusion follows from the fact that the transportation cost can be written as
\begin{equation*}
    \|\tilde{L}_k^\dagger \xi\|^{2} = \sum_{i=1}^{Nd}\frac{1}{\sigma_i^2}|\eta_i^\intercal \xi|^2 \leq \sigma_{\min}^{2}(\tilde{L}_k^\dagger),
\end{equation*}
where $\sigma_i$ are the singular values of $\tilde{L}_k$ and $\eta_i$ are the columns of $U$ resulting from the SVD $\tilde{L}_k = U\Sigma V^\intercal$, and the second inclusion results from Theorem~\ref{thm:Gelbrich_inclusion}.
Thus, it suffices to satisfy the Gelbrich DR-CVaR constraints
\begin{equation*}
    \sup_{\P\in\mathcal{G}_{\bar\varepsilon}(\hat\mu,\hat\Sigma)} \ \cvar_{1-\gamma}^{\P}(\alpha^\intercal x + \beta) \leq 0,
\end{equation*}
where $\bar\varepsilon \triangleq \varepsilon\sigma_{\max}^{2}(\tilde{L}_k)$.
By definition, the Gelbrich ambiguity set contains all distributions in $\mathcal{S}$ whose mean vectors and covariance matrices belong to $\mathcal{U}_{\bar\varepsilon}(\hat\mu,\hat\Sigma)$.
It is fairly straightforward to show that we can equivalently write this ambiguity set as
\begin{equation}
    \label{eq:Gelbrich_decomposition}
    \mathcal{G}_{\bar\varepsilon}(\hat\mu,\hat\Sigma) = \bigcup_{(\mu,\Sigma)\in\mathcal{U}_{\bar\varepsilon}(\hat\mu,\hat\Sigma)} \mathcal{C}(\mu,\Sigma),
\end{equation}
where $\mathcal{C}(\mu,\Sigma)$ is the (structured) Chebyshev ambiguity set that contains all distribution in $\mathcal{S}$ with mean $\mu$ and covariance $\Sigma$.
Further, using \eqref{eq:Gelbrich_decomposition}, the DR-CVaR risk can be decomposed as
\begin{equation}
    \label{eq:CVaR_decomposition}
    \sup_{\P\in\mathcal{G}_{\bar\varepsilon}(\hat\mu,\hat\Sigma)} \cvar_{1-\gamma}^{\P}(\ell) = \sup_{(\mu,\Sigma)\in\mathcal{U}_{\bar\varepsilon}(\hat\mu,\hat\Sigma)}\sup_{\P\in\mathcal{C}(\mu,\Sigma)}\cvar_{1-\gamma}^{\P}(\ell),
\end{equation}
where $\ell\triangleq \alpha^\intercal x + \beta$ for our problem.
The innermost maximization in \eqref{eq:CVaR_decomposition} can be reformulated as follows
\begin{align}
    &\sup_{\P\in\mathcal{C}(\mu,\Sigma)}\cvar_{1-\gamma}^{\P}(\alpha^\intercal x + \beta) \nonumber \\
    &= \beta + \alpha^\intercal \mu + \sup_{\P\in\mathcal{C}(\mu,\Sigma)} \cvar_{1-\gamma}^{\P}\big(\alpha^\intercal (x - \mu)\big) \nonumber \\
    &= \beta + \alpha^\intercal\mu + \sqrt{\alpha^\intercal\Sigma\alpha}\sup_{\P\in\mathcal{C}(\mu,\Sigma)}\cvar_{1-\gamma}^{\P}\left[\frac{\alpha^\intercal(x - \mu)}{\sqrt{\alpha^\intercal\Sigma\alpha}}\right] \nonumber \\
    &= \beta + \alpha^\intercal \mu + \tau(\mu,\Sigma,\alpha) \sqrt{\alpha^\intercal \Sigma \alpha},
\end{align}
where in the first and second equalities, we use the fact that the CVaR risk measure is translation invariant and positive homogeneous \cite{DRO_Kuhn_risk}, and in the last equality we use the definition of the standard risk coefficient.
It can be shown \cite{DRO_Kuhn_risk} that given the structural ambiguity set $\mathcal{S} = \mathcal{P}_{2}(\R^{n})$, and for the CVaR risk measure, the corresponding standard risk coefficient $\tau$ is \textit{independent} of $\mu,\Sigma$, and $\alpha$.
As a result, the optimization problem \eqref{eq:CVaR_decomposition} can be reformulated as
\begin{equation}~\label{eq:DR_CVaR_reformulated}
    \begin{aligned}
        &\sup_{\mu,\Sigma\succeq 0} &&\beta + \alpha^\intercal \mu + \tau \sqrt{\alpha^\intercal \Sigma\alpha} \\
        &\ \ \mathrm{s.t.} &&\|\mu - \hat\mu\|^2 + \mathrm{tr}[\Sigma + \hat\Sigma - 2(\hat\Sigma^{\frac{1}{2}}\Sigma\hat\Sigma^{\frac{1}{2}})^{\frac{1}{2}}] \leq \bar\varepsilon^2.
    \end{aligned}
\end{equation}
Taking the dual of the maximization problem \eqref{eq:DR_CVaR_reformulated} yields
\begin{align}
    &\inf_{\lambda \geq 0}\sup_{\mu,\Sigma\succeq 0} \beta + \alpha^\intercal \mu + \tau \sqrt{\alpha^\intercal\Sigma\alpha} + \lambda\big(\bar\varepsilon^2 - \|\mu-\hat\mu\|^2 \nonumber \\
    &\hspace{4cm}- \mathrm{tr}[\Sigma + \hat\Sigma - 2(\hat\Sigma^{\frac{1}{2}}\Sigma\hat\Sigma^{\frac{1}{2}})^{\frac{1}{2}}]\big) \nonumber \\
    =&\inf_{\lambda\geq 0}\bigg\{\beta + \lambda(\bar\varepsilon^2 - \mathrm{tr}[\hat\Sigma] + \sup_{\mu}\big\{\alpha^\intercal \mu - \lambda \|\mu - \hat\mu\|^2\big\} \nonumber \\
    &\hspace{1cm}+ \sup_{\Sigma\succeq 0}\big\{\tau\sqrt{\alpha^\intercal\Sigma\alpha} + \lambda\mathrm{tr}[-\Sigma + 2(\hat\Sigma^{\frac{1}{2}}\Sigma\hat\Sigma^{\frac{1}{2}})^{\frac{1}{2}}]\big\}\bigg\}. \label{eq:DR_CVaR_dual}
\end{align}
The first supremum over $\mu$ is a simple quadratic maximization problem over a concave function and can be solved analytically, which yields the maximizer $\mu^\star = \frac{\alpha}{2\lambda} + \hat\mu$ with optimal value $\frac{\|\alpha\|^2}{4\lambda} + \alpha^\intercal\hat\mu$.
The second supremum over $\Sigma$ can be reformulated by introducing the auxiliary epigraphical variable $t$ via
\begin{equation}
    \label{eq:second_supremum}
    \sup_{\Sigma\succeq 0, t \geq 0} \tau t + \lambda \mathrm{tr}[-\Sigma + 2(\hat\Sigma^{\frac{1}{2}}\Sigma\hat\Sigma^{\frac{1}{2}})^{\frac{1}{2}}] \quad \mathrm{s.t.} \quad  t^2 \leq \alpha^\intercal\Sigma\alpha.
\end{equation}
Now introduce the variable substitution $B \triangleq (\hat\Sigma^{\frac{1}{2}}\Sigma\hat\Sigma^{\frac{1}{2}})^{\frac{1}{2}}$ and taking the dual of the maximization problem \eqref{eq:second_supremum} yields
\begin{align}
    &\inf_{\rho \geq 0} \sup_{\Sigma\succeq 0, t\geq 0} \tau t - \lambda\mathrm{tr}[\Sigma] + 2\lambda\mathrm{tr}[(\hat\Sigma^{\frac{1}{2}}\Sigma\hat\Sigma^{\frac{1}{2}})^{\frac{1}{2}}] \nonumber \\
    &\hspace{5cm} + \rho(\alpha^\intercal\Sigma\alpha - t^2) \nonumber \\
    =&\inf_{\rho\geq 0}\sup_{\Sigma\succeq 0, t\geq 0} \tau t - \rho t^2 + \mathrm{tr}[\Sigma(\rho \alpha\alpha^\intercal - \lambda I)] \nonumber \\
    &\hspace{5cm} + 2\lambda\mathrm{tr}[(\hat\Sigma^{\frac{1}{2}}\Sigma\hat\Sigma^{\frac{1}{2}})^{\frac{1}{2}}] \nonumber \\
    =&\inf_{\rho\geq 0}\sup_{B\succeq 0, t\geq 0} \tau t - \rho t^2 + \mathrm{tr}[B^2\Delta_{\rho}] + 2\lambda\mathrm{tr}[B] \nonumber \\
    =&\inf_{\rho\geq 0}\bigg\{\sup_{t\geq 0}\big\{\tau t - \rho t^2\big\} + \sup_{B\succeq 0}\big\{\mathrm{tr}[B^2\Delta_{\rho}] + 2\lambda\mathrm{tr}[B]\big\}\bigg\}, \label{eq:second_supremum_dual}
\end{align}
where $\Delta_{\rho} \triangleq \hat\Sigma^{-\frac{1}{2}}(\rho\alpha\alpha^\intercal - \lambda I)\hat\Sigma^{-\frac{1}{2}}$ for any $\rho \geq 0$.
The first supremum has maximizer $t^\star = \tau/2\rho$ and optimal value $\tau^2/4\rho$, and the second supremum is tractable under the assumption that $\Delta_{\rho} \prec 0$, or equivalently, $\lambda \|\alpha\|^{-2} > \rho$.
Taking the first order necessary conditions yields the condition
\begin{equation*}
    B^\star \Delta_{\rho} + \Delta_{\rho}B^\star + 2\lambda I = 0,
\end{equation*}
which yields the maximizer $B^\star = -\lambda\Delta_{\rho}^{-1} \succ 0$ and the optimal value $-\lambda^2\mathrm{tr}[\Delta_{\rho}^{-1}]$.
This maximizer is \textit{unique} because the necessary condition can be interpreted as a Lyapunov equation, whose solution is unique if and only if $\Delta_{\rho}$ is Hurwitz \cite{Hespanha09}.
The dual minimization problem \eqref{eq:second_supremum_dual} thus becomes
\begin{equation}
    \label{eq:second_supremum_dual_reformulated}
    \inf_{0 < \rho < \lambda \|\alpha\|^{-2}} \ \frac{\tau^2}{4\rho} + \lambda^2 \mathrm{tr}[\hat\Sigma^{\frac{1}{2}}(\lambda I - \rho \alpha\alpha^\intercal)^{-1}\hat\Sigma^{\frac{1}{2}}].
\end{equation}
Next, using the Sherman-Morrison formula, we can express \eqref{eq:second_supremum_dual_reformulated} as
\begin{align*}
    &\inf_{0 < \rho < \lambda \|\alpha\|^{-2}} \ \frac{\tau^2}{4\rho} + \lambda\mathrm{tr}[\hat\Sigma] + \frac{\alpha^\intercal\hat\Sigma\alpha}{\rho^{-1} - \|\alpha\|^2/\lambda} \\
    &= \lambda\mathrm{tr}[\hat\Sigma] + \frac{\tau^2}{4}\frac{\|\alpha\|^2}{\lambda} + \tau\sqrt{\alpha^\intercal\hat\Sigma\alpha}.
\end{align*}
In summary, the dual of the DR-CVaR risk \eqref{eq:DR_CVaR_dual} becomes
\begin{align}
    &\inf_{\lambda\geq 0}  \ \beta + \lambda(\bar\varepsilon^2 - \mathrm{tr}[\hat\Sigma]) + \frac{\|\alpha\|^2}{4\lambda} + \alpha^\intercal\hat\mu + \lambda\mathrm{tr}[\hat\Sigma] \nonumber \\
    &\hspace{3.5cm}+ \frac{\tau^2}{4}\frac{\|\alpha\|^2}{\lambda} + \tau \sqrt{\alpha^\intercal\hat\Sigma\alpha} \nonumber \\
    =&\inf_{\lambda\geq 0} \ \beta + \alpha^\intercal\hat\mu + \tau\sqrt{\alpha^\intercal\hat\Sigma\alpha} + \lambda\bar\varepsilon^2 + \frac{\tau^2 + 1}{4}\frac{\|\alpha\|^2}{\lambda} \nonumber \\
    =&\beta + \alpha^\intercal\hat\mu + \tau\sqrt{\alpha^\intercal\hat\Sigma\alpha} + \bar\varepsilon\sqrt{1 + \tau^2} \|\alpha\|, \label{eq:DR_CVaR_final}
\end{align}
which achieves the desired result.
\section*{B.~Proof of Corollary~\ref{cor:CVaR_tractable_constraints}}
\setcounter{equation}{0}
\renewcommand{\theequation}{B.\arabic{equation}}
Plugging in the propagated nominal distribution covariance $\hat\Sigma_k = E_k(I + \mathcal{B}L)\mathcal{D}\Sigma_{w}\mathcal{D}^\intercal(I+\mathcal{B}L)^\intercal E_k^\intercal$
into the DR-CVaR constraints \eqref{eq:DR_CVaR_final} yields the constraints
\begin{align}
    &\beta + \alpha^\intercal\hat\mu + \tau\|\Sigma_w^{1/2}\mathcal{D}^\intercal (I + \mathcal{B}L)^\intercal E_k^\intercal\alpha\| \nonumber \\
    &\hspace{1cm} + \varepsilon \|\alpha\| \sqrt{1 + \tau^2} \sigma_{\max}^{2}(E_k(I + \mathcal{B}L)\mathcal{D}) \leq 0. \label{eq:DR_CVaR_control}
\end{align}
Introducing the epigraphical variable $\rho$ for the last term in \eqref{eq:DR_CVaR_control}, we get
\begin{align}
    &\beta + \alpha^\intercal\hat\mu + \tau\|\Sigma_w^{1/2}\mathcal{D}^\intercal (I + \mathcal{B}L)^\intercal E_k^\intercal\alpha\| \nonumber \\
    &\hspace{4cm} + \varepsilon \rho \|\alpha\| \sqrt{1 + \tau^2} \leq 0, \label{eq:CVaR_SOCC} \\
    &\rho \geq \sigma_{\max}^{2}(E_k(I + \mathcal{B}L)\mathcal{D}). \label{eq:CVaR_SVC}
\end{align}
The first constraint \eqref{eq:CVaR_SOCC} is a second-order cone constraint, and thus amenable to off-the-shelf convex solvers \cite{MOSEK}.
For the second constraint \eqref{eq:CVaR_SVC}, note the equivalence $\sigma_{\max}^{2}(A) = \lambda_{\max}(A^\intercal A) \leq \rho \iff A^\intercal A \leq \rho I$, which is equivalent to the LMI
\begin{equation}
    \begin{bmatrix}
    I & A \\
    A^\intercal & \rho I
    \end{bmatrix} \succeq 0.
\end{equation} 
Applying this reasoning to $\sigma_{\max}^{2}(\tilde{L}_k) \leq \rho$ achieves the desired result.
\section*{C.~Proof of Theorem~\ref{thm:DR_cost_convex}}
\setcounter{equation}{0}
\renewcommand{\theequation}{C.\arabic{equation}}
Similar to Appendix~A, we can decompose the worst-case risk over the Gelbrich ambiguity set as a supremum over a Chebyshev ambiguity set $\mathcal{C}(\mu,\Sigma)$ embedded in a supremum over the uncertainty set $\mathcal{U}_{\varepsilon}(\hat\mu,\hat\Sigma)$.
The maximum value of the expected value of a quadratic loss function over a Chebyshev ambiguity set is simply given by
\begin{equation*}
    \sup_{\P\in\mathcal{C}(\mu,\Sigma)} \ \E_{\P}[\vct w^\intercal \Xi \vct w] = \mu^\intercal\Xi \mu + \mathrm{tr}[\Xi\Sigma].
\end{equation*}
Thus, the Gelbrich DR objective simplifies to the maximization problem
\begin{equation*}
    \sup_{\mu,\Sigma\succeq 0} \mu^\intercal\Xi \mu + \mathrm{tr}[\Xi\Sigma] \quad \mathrm{s.t.} \quad \G^{2}\big((\mu,\Sigma), (\hat\mu,\hat\Sigma)\big) \leq \varepsilon^2,
\end{equation*}
or equivalently, by using the definition of the Gelbrich distance, as
\begin{equation}~\label{eq:Gelbrich_DR_cost}
    \begin{aligned}
        &\sup_{\mu} \sup_{\Sigma\succeq 0} &&\mu^\intercal\Xi \mu + \mathrm{tr}[\Xi\Sigma] \\
        &\quad\mathrm{s.t.} &&\|\mu-\hat\mu\|^2 + \mathrm{tr}[\Sigma + \hat\Sigma - 2(\hat\Sigma^{\frac{1}{2}}\Sigma\hat\Sigma^{\frac{1}{2}})^{\frac{1}{2}}] \leq \varepsilon^2.
    \end{aligned}
\end{equation}
Next, we take the dual of \eqref{eq:Gelbrich_DR_cost} with respect to the maximization over $\Sigma$, which yields
\begin{align}
    &\sup_{\mu} \inf_{\gamma \geq 0}\sup_{\Sigma\succeq 0} \mu^\intercal\Xi \mu + \mathrm{tr}[\Xi\Sigma] + \gamma\big(\varepsilon^2 - \|\mu - \hat\mu\|^2 - \mathrm{tr}[\Sigma \nonumber \\
    &\hspace{4.5cm}+ \hat\Sigma - 2(\hat\Sigma^{\frac{1}{2}}\Sigma\hat\Sigma^{\frac{1}{2}})^{\frac{1}{2}}]\big) \nonumber \\
    =&\sup_{\mu} \bigg\{\mu^\intercal \Xi \mu + \inf_{\gamma \geq 0} \Big\{\gamma \varepsilon^2 - \gamma \|\mu - \hat\mu\|^2 - \gamma \mathrm{tr}[\hat\Sigma] \nonumber \\
    &\hspace{1cm}+ \sup_{\Sigma\succeq 0}\big\{\mathrm{tr}[\Xi\Sigma] - \gamma\big(\mathrm{tr}[\Sigma - 2(\hat\Sigma^{\frac{1}{2}}\Sigma\hat\Sigma^{\frac{1}{2}})^{\frac{1}{2}}]\big)\big\}\Big\}\bigg\}. \label{eq:Gelbrich_DR_cost_dual}
\end{align}
Assuming $\gamma > \lambda_{\max}(\Xi)$ and $\hat\Sigma \succ 0$, the inner maximization over $\Sigma$ in \eqref{eq:Gelbrich_DR_cost_dual} can be solved analytically as follows.
First, and similar to Appendix~A, let $B \triangleq (\hat\Sigma^{\frac{1}{2}}\Sigma\hat\Sigma^{\frac{1}{2}})^{\frac{1}{2}}$, which implies $\Sigma = \hat\Sigma^{-\frac{1}{2}}B^2\hat\Sigma^{-\frac{1}{2}}$.
The inner maximization then becomes
\begin{align*}
    &\sup_{B \succeq 0} \ \mathrm{tr}[\Xi\hat\Sigma^{-\frac{1}{2}}B^2\hat\Sigma^{-\frac{1}{2}}] - \gamma \mathrm{tr}[\hat\Sigma^{-\frac{1}{2}}B^2\hat\Sigma^{-\frac{1}{2}} - 2B] \\
    =&\sup_{B\succeq 0} \ \mathrm{tr}[B^2\hat\Sigma^{-\frac{1}{2}}(\Xi - \gamma I)\hat\Sigma^{-\frac{1}{2}}] + 2\gamma\mathrm{tr}[B] \\
    =&\sup_{\B \succeq 0} \ \mathrm{tr}[B^2\Delta_{\gamma}] + 2\gamma\mathrm{tr}[B],
\end{align*}
where $\Delta_{\gamma}\triangleq \hat\Sigma^{-\frac{1}{2}}(\Xi - \gamma I)\hat\Sigma^{-\frac{1}{2}}$, for any $\gamma \geq 0$.
Since $\gamma > \lambda_{\max}(\Xi)$ and $\hat\Sigma \succ 0$, it follows that $\Delta_{\gamma} \preceq 0$, and thus the objective function is concave in $B$.
To this end, the maximizer becomes $B^\star = -\gamma\Delta_{\gamma}^{-1}$, resulting from the first order optimality conditions
\begin{equation*}
    B^\star \Delta_{\gamma} + \Delta_{\gamma} B^\star + 2\gamma I = 0,
\end{equation*}
and is unique because the associated Lyapunov equation has a unique solution if and only if $\Delta_{\gamma}$ is Hurwitz \cite{Hespanha09}.
Plugging this back into the objective function gives the optimal value $J^\star(\hat\Sigma) = \gamma^2\mathrm{tr}[\hat\Sigma(\Xi-\gamma I)^{-1}]$, with associated maximizer $\Sigma^\star(\hat\Sigma) = \gamma^2(\gamma I - \Xi)^{-1}\hat\Sigma(\gamma I - \Xi)^{-1}$, which holds when $\gamma > \lambda_{\max}(\Xi)$ and $\hat\Sigma \succ 0$.
The cases when $\gamma\not\succ\lambda_{\max}(\Xi)$ and $\hat\Sigma\succeq 0$ may also be treated in a similar manner using $\lim\inf$ arguments, see [\cite{DRO_Kuhn_MSE}, Proposition A.3] for a detailed analysis.
Plugging the optimal value of the maximization over $\Sigma$ back into the dual formulation \eqref{eq:Gelbrich_DR_cost_dual} yields the optimization problem
\begin{align}
    &\sup_{\mu} \inf_{\gamma > \lambda_{\max}(\Xi)} \mu^\intercal \Xi \mu + \gamma \big(\varepsilon^2 - \|\mu-\hat\mu\|^2 - \mathrm{tr}[\hat\Sigma]\big) \nonumber \\
    &\hspace{4.5cm} + \gamma^2 \mathrm{tr}[\hat\Sigma(\Xi - \gamma I)^{-1}] \nonumber \\
    =&\inf_{\gamma I \succ \Xi}\bigg\{\sup_{\|\mu-\hat\mu\|\leq\varepsilon} \Big\{\mu^\intercal\Xi\mu - \gamma\|\mu-\hat\mu\|^2\Big\} + \gamma\big(\varepsilon^2 - \mathrm{tr}[\hat\Sigma] \nonumber \\
    &\hspace{3.5cm} + \gamma \mathrm{tr}[\hat\Sigma(\Xi - \gamma I)^{-1}]\big)\bigg\}, \label{eq:Gelbrich_DR_cost_dual_reformulated}
\end{align}
where the equality holds because the Gelbrich constraint is infeasible unless $\|\mu-\hat\mu\| \leq \varepsilon$ and we can switch the $\sup$ and $\inf$ due to the minimax theorem \cite{Sion1958OnGM}, which applies because $\mu$ ranges over a compact ball and because $\gamma I - \Xi \succ 0$.
For the inner maximization over $\mu$, we complete the square to obtain the objective function $(\mu - z)^\intercal(\Xi - \gamma I)(\mu - z) + y$, where
\begin{align*}
    &z = \gamma(\gamma I - \Xi)^{-1}\hat\mu, \quad y = \hat\mu^\intercal P \hat\mu, \\
    &P = -\gamma I + \gamma^2(\gamma I - \Xi)^{-1}.
\end{align*}
Thus, the maximizer is $\mu^\star = z$ and the optimal value is $y$.
Plugging this back into \eqref{eq:Gelbrich_DR_cost_dual_reformulated} yields the convex program
\begin{align}
    &\inf_{\gamma I \succ \Xi} \gamma(\varepsilon^2 - \mathrm{tr}[\hat\Sigma] - \|\hat\mu\|^{2}) + \gamma^2\big(\hat\mu^\intercal (\gamma I - \Xi)^{-1}\hat\mu \nonumber \\
    &\hspace{4.4cm}+ \mathrm{tr}[\hat\Sigma(\gamma I - \Xi)^{-1}]\big).
\end{align}
Lastly, since we assume that the nominal noise ambiguity set is zero mean, then $\hat\mu = 0$, which achieves the desired result.

\section*{D.~Proof of Corollary~\ref{cor:DR_cost_tractable}}
\setcounter{equation}{0}
\renewcommand{\theequation}{D.\arabic{equation}}
The constraints in \eqref{eq:convex_DR_cost} are convex, but nonlinear in the decision variables $\gamma$ and $L$ due to the last term.
To this end, consider the function
\begin{equation*}
    h(\Xi, \lambda) \triangleq 
    \begin{cases}
        \lambda^2 \mathrm{tr}[\hat\Sigma(\lambda I - \Xi^{-1}], &\mathrm{if} \ \lambda I - \Xi \succ 0, \\
        \infty, &\mathrm{otherwise}.
    \end{cases}
\end{equation*}
If $\lambda I - \Xi \succ 0$, then we have the following equivalence
\begin{align}
    h(\Xi, \lambda) &= \inf_{\Gamma \succeq 0} \ \mathrm{tr}[\Gamma] \quad \mathrm{s.t.} \quad \Gamma \succeq \lambda^2 \hat\Sigma^{\frac{1}{2}} (\lambda I - \Xi)^{-1}\hat\Sigma^{\frac{1}{2}} \nonumber \\
    &= \inf_{\Gamma \succeq 0, \Psi \succ 0} \ \mathrm{tr}[\Gamma] \quad \mathrm{s.t.}
    \begin{aligned}~\label{eq:DR_cost_SDP_reformulation}
         \quad &\Gamma \succeq \lambda^2 \hat\Sigma^{\frac{1}{2}} \Psi^{-1} \hat\Sigma^{\frac{1}{2}}, \\
        &\lambda I - \Xi \succeq \Psi,
    \end{aligned}
\end{align}
where the first equality holds from introducing the auxiliary variable $\Gamma$ and noting that $\Gamma \succeq \bar\Gamma$ implies $\mathrm{tr}[\Gamma] \geq \mathrm{tr}[\bar\Gamma]$ for all $\Gamma,\bar\Gamma \succeq 0$, and the cyclic property of the trace operator.
Similarly, the second equality holds from introducing another auxiliary variable $\Psi$ and noting that $\Psi \succeq \bar\Psi$ is equivalent to $\Psi^{-1} \preceq \bar\Psi^{-1}$ for all $\Psi, \bar\Psi \succ 0$.
The first constraint in \eqref{eq:DR_cost_SDP_reformulation} can be equivalently written as the SDP \eqref{eq:DR_cost_constraint1}.
For the second constraint in \eqref{eq:DR_cost_SDP_reformulation}, we expand $\Xi(L)$, which gives $\Xi(L) = \mathcal{D}^\intercal (\tilde{M}(L) + L^\intercal\mathcal{R}L)\mathcal{D}$, where $\tilde{M}(L)$ as defined in \eqref{eq:tractable_DR_cost} is linear in $L$.
Thus, the resulting constraints become
\begin{equation*}
    \lambda I - \mathcal{D}^\intercal \tilde{M}(L) \mathcal{D} - \mathcal{D}^\intercal L^\intercal \tilde{R} L \mathcal{D} \succeq \Psi,
\end{equation*}
which can be equivalently written as the SDP \eqref{eq:DR_cost_constraint2}, where $\tilde{R}$ is invertible because $\mathcal{R} \succ 0$.
This concludes the proof.

\section*{E.~Dryden Turbulence Model}
\setcounter{equation}{0}
\renewcommand{\theequation}{E.\arabic{equation}}
The Dryden turbulence model is a zero-mean, stationary Gaussian process model for wind gusts, characterized by the power spectral density (PSD) along each linear and angular velocity channel.
In terms of the frequency $\omega$, the PSD along each linear velocity channel is given by
\begin{equation}
    \begin{aligned}
        \Phi_{u_{g}}(\omega) &= \frac{2\sigma_u^2 L_u}{\pi V_0}\frac{1}{1 + (L_u\omega / V_0)^2}, \\
        \Phi_{v_{g}}(\omega) &= \frac{2\sigma_v^2 L_v}{\pi V_0}\frac{1 + 12 (L_v\omega / V_0)^2}{(1 + 4(L_v\omega / V_0)^2)^2}, \\
        \Phi_{w_{g}}(\omega) &= \frac{2\sigma_{w}^2 L_w}{\pi V_0}\frac{1 + 12 (L_w\omega / V_0)^2}{(1 + 4(L_w\omega / V_0)^2)^2},
    \end{aligned}
\end{equation}
where $V_0$ denotes the mean wind speed at 20 feet altitude.
The turbulence intensities $\{\sigma_i, i = 1,2,3\}$ at low altitudes are computed from
\begin{equation}
    \begin{aligned}
        \sigma_{w} &= 0.1 V_0, \\
        \sigma_{u} &= \frac{\sigma_{w}}{(0.177 + 0.000823 z)^{0.4}}, \\
        \sigma_{v} &= \sigma_{u}.
    \end{aligned}
\end{equation}
The characteristic length scales $\{L_u, L_v, L_w, i = 1,2,3\}$ at low altitudes are computed from
\begin{equation}
    \begin{aligned}
        L_{u} &= \frac{z}{(0.177 + 0.000823z)^{1.2}}, \\
        L_{v} &= L_{u}, \\
        L_{w} &= z,
    \end{aligned}
\end{equation}
where $z$ denotes the altitude.
Similarly, the PSD along each angular velocity channel is given by
\begin{equation}
    \begin{aligned}
        \Phi_{p_{g}}(\omega) &= \frac{\sigma_w^2}{2V_0 L_w} \frac{0.8\left(\frac{2\pi L_w}{4b}\right)^{1/3}}{1 + \left(\frac{4b\omega}{\pi V_0}\right)^2}, \\ 
        \Phi_{q_g}(\omega) &= \frac{(\omega/V_0)^2}{1 + \left(\frac{4b\omega}{\pi V_0}\right)^2} \Phi_{w_g}(\omega), \\
        \Phi_{r_g}(\omega) &= \frac{(\omega/V_0)^2}{1 + \left(\frac{3b\omega}{\pi V_0}\right)^2}\Phi_{v_g}(\omega),
    \end{aligned}
\end{equation}
where $b$ denotes the span of the quadcopter.
For simplicity, we choose a reference altitude $z = 10$ m, and a wingspan $b = 0.34$ m, typical of small quadcopter.

\end{document}